\documentclass[reqno,centertags,11pt]{amsart}
\usepackage{verbatim}
\usepackage{tikz, enumerate, array, amssymb, amsmath}
\usepackage{graphicx, epsfig }
\usepackage{amsthm}
\usepackage{amsbsy}
\usepackage{amsfonts}
\usepackage[colorlinks]{hyperref}

\usepackage{mathtools}

\newcommand{\Hmm}[1]{\leavevmode{\marginpar{\tiny%
$\hbox to 0mm{\hspace*{-0.5mm}$\leftarrow$\hss}%
\vcenter{\vrule depth 0.1mm height 0.1mm width \the\marginparwidth}%
\hbox to 0mm{\hss$\rightarrow$\hspace*{-0.5mm}}$\\\relax\raggedright #1}}}

\newcommand{\N}{{\mathbb{N}}}

\newcommand{\R}{{\mathbb{R}}}

\newcommand{\C}{{\mathbb{C}}}

\newcommand{\blue}{\color{blue} }

\newcommand{\f}{\frac}

\newcommand{\ol}{\overline}

\newcommand{\wti}{\widetilde}

\newcommand{\oh}{o}

\newcommand{\loc}{\text{\rm{loc}}}

\newcommand{\hatt}{\widehat}
\newcommand{\beq}{\begin{equation}}
\newcommand{\eeq}{\end{equation}}
\newcommand{\bdm}{\begin{displaymath}}
\newcommand{\edm}{\end{displaymath}}
\newcommand{\ba}{\begin{align}}
\newcommand{\ea}{\end{align}}
\newcommand{\bpf}{\begin{proof}}
\newcommand{\epf}{\end{proof}}

\newcommand{\la}{\langle}
\newcommand{\ra}{\rangle}

\newcommand{\supp}{\mathrm{supp}\, }               

\newcommand{\veps}{\varepsilon}

\newcommand{\re}{\mathrm{Re}}
\newcommand{\im}{\mathrm{Im}}

\newcommand{\dav}{{d_{\mathrm{av}}}}

\newcommand{\id}{\mathbf{1}}                

\allowdisplaybreaks

\newcommand{\calC}{\mathcal{C}}

\newcommand{\calQ}{\mathcal{Q}}

\newcommand{\calS}{{\mathcal{S}}}

\newtheorem{theorem}{Theorem}
\newtheorem{proposition}[theorem]{Proposition}
\newtheorem{lemma}[theorem]{Lemma}
\newtheorem{corollary}[theorem]{Corollary}

\theoremstyle{definition}

\newtheorem{remark}[theorem]{Remark}
\newtheorem{remarks}[theorem]{Remarks}

\newcounter{theoremi}[theorem]

\newcommand{\itemthm}{\refstepcounter{theoremi} {\rm(\roman{theoremi})}{~}}

\numberwithin{theorem}{section}
\numberwithin{equation}{section}

\newcounter{assumptions}
\newenvironment{assumptions}{\begin{list}{\textbf{A\rm\arabic{assumptions}})}{%
\setlength{\topsep}{0mm}\setlength{\parsep}{0mm}\setlength{\itemsep}{0mm}%
\setlength{\labelwidth}{2em}\setlength{\leftmargin}{2em}\usecounter{assumptions}%
}}{\end{list}}

\newcounter{saveassumptions}
\newcounter{smalllist}

\newcounter{listi}
\newenvironment{theoremlist}{\begin{list}{{\rm(\roman{listi})}}{%
\setlength{\topsep}{0mm}\setlength{\parsep}{0mm}\setlength{\itemsep}{0mm}%
\setlength{\labelwidth}{1.5em}\setlength{\leftmargin}{1.7em}\usecounter{listi}%
}}{\end{list}}

\newcounter{smallenum}

\begin{document}

\title[Well--posedness of DMNLS]{Well--posedness of dispersion managed nonlinear Schr\"odinger equations}
\author[M.--R.~Choi, D.~Hundertmark,  Y.--R.~Lee]{Mi--Ran Choi, Dirk Hundertmark,  Young--Ran~Lee}
\address{Department of Mathematics, Sogang University,  35 Baekbeom-ro (Sinsu--dong),
    Mapo-gu, Seoul, 04107, South Korea.}%
\email{mrchoi@sogang.ac.kr}
\address{Department of Mathematics, Institute for Analysis, Karlsruhe Institute of Technology, 76128 Karlsruhe, Germany, and
Department of Mathematics
University of Illinois at Urbana-Champaign
1409 W. Green Street
Urbana, Illinois 61801-2975  }%
\email{dirk.hundertmark@kit.edu}%
\address{Department of Mathematics, Sogang University,  35 Baekbeom-ro (Sinsu--dong),
    Mapo-gu, Seoul, 04107, South Korea.}%
\email{younglee@sogang.ac.kr}

\thanks{\textit{MSC2020 classification}. 35Q55, 35Q60, 35A01, 35B35, 35B30.}
\thanks{\copyright 2022 by the authors. Faithful reproduction of this article,
       in its entirety, by any means is permitted for non-commercial purposes}

\keywords{nonlocal NLS, dispersion management, well--posedness, orbital stability}

\date{\today}

\begin{abstract}
  We prove local and global well--posedness results for
  the Gabitov--Turitsyn or dispersion managed nonlinear
  Schr\"odinger equation with a large class of nonlinearities and
  arbitrary average dispersion on $L^2(\R)$ and $H^1(\R)$ for
  zero and non--zero average dispersions, respectively.
  Moreover,  when the average dispersion is non--negative,
  we show that the set of ground states is orbitally stable.
  This covers the case of
  non--saturated and saturated nonlinear polarizations and yields,
  for saturated nonlinearities, the first proof of orbital stability.
\end{abstract}

\maketitle
{\hypersetup{linkcolor=black}
\setcounter{tocdepth}{1}
\tableofcontents}

\section{Introduction}\label{introduction}
\subsection{The Cauchy problem}
We prove local and global existence results for the initial value
problem for a dispersion managed nonlinear Schr\"odinger equation (NLS)
\beq\label{eq:main}
\begin{cases}
   \displaystyle{i\partial_t u+ \dav \partial_x^2u+ \int _\R T_r^{-1}(P(T_r u))\psi(r)dr=0},\\
   u(x,0)=u_0(x) ,
\end{cases}
\eeq
for a large class of nonlinear polarizations $P$
where $u=u(x,t)$, $x, t\in \R$, is a complex-valued function,  $\dav \in \R$, $\psi \ge 0$ is in $L^q(\R)$ for suitable $q \ge 1$, and $T_r=e^{ir\partial_x^2}$ is the solution operator for the free Schr\"{o}dinger equation, that is,  $w(x,r)=(T_rf)(x)$ solves the initial value problem
\bdm
   \begin{cases}
   i\partial_r w+ \partial_x^2w=0,\\
   w(x,0)=f(x).
\end{cases}
\edm
The case $\dav=0$ is a \emph{singular limit}.
Positive average dispersion, $\dav>0$, corresponds to a
\emph{focusing} nonlinearity, while $\dav<0$ corresponds to a \emph{defocusing} nonlinearity, see Remark \ref{rem:focusing-defocusing}.
In the application of \eqref{eq:main} in nonlinear optics, $t$ corresponds to the distance along the fiber and $x$ denotes the (retarded) time.
The name ``dispersion management" refers to the fact that the equation
\eqref{eq:main} models the propagation of signals through glass--fiber
cables where the local dispersive properties vary periodically between
strongly positive and strongly negative dispersion, with some small
average dispersion $\dav$, along the cable. It is an effective
equation  describing the electromagnetic wave propagation in optical
fibers in the
so--called \emph{strong dispersion management regime}.
See Section  \ref{sec:connection} for a short discussion on how the
probability density $\psi$ is determined from the
local dispersion profile in dispersion managed glass fiber cables.

The technique of dispersion management was invented to balance
the competing effects of nonlinearity and dispersion. It has led to
new type of glass--fiber cables for ultra--high speed data transfer through optical fiber over long distances.
The dispersion managed NLS has intensively been studied, mainly on a
non--rigorous level starting with \cite{AB98,GT96a,GT96b}, see also
the survey \cite{TBF}  and references therein. There are much fewer rigorous
results available, e.g.,
\cite{ChoiHuLee2016, EHL2009, GH, HuLee2009, HuLee2012, MH1, Stanislavova05, ZGJT01}. 
The Kerr--type nonlinearity, i.e.,  the case when $P(z)= |z|^2z$, was originally studied by Gabitov and Turitsyn in \cite{GT96a, GT96b} and is
also assumed in much of the rigorous and non--rigorous work.
It corresponds to taking a Taylor series expansion of the polarization $P$
in the glass--fiber cable and keeping the only the first non--trivial term.
We will not make this simplifying assumption in our paper but consider a
rather large class of nonlinearities, instead.

We assume that the nonlinearity $ P: \C \to \C$ in \eqref{eq:main} is of
the form of $P(z)= h(|z|)z$.
Our main assumptions on $h: [0,\infty)\to\R $ are:

\begin{assumptions}
\item \textbf{Local and global well--posedness, zero average dispersion:}\label{ass: dav zero}
    $h$ is continuous on $[0,\infty)$ and continuously differentiable on
    $(0,\infty)$ with $\lim_{a\to 0}h'(a)a=0$. There exists
    $0\le  p \le 4$ such that
    \begin{equation} \label{eq:ass dav zero} \
     \begin{split}
	   |h(a)|&\lesssim 1+a^{p}\, \quad\text{for all }  a\ge 0, \\
	   |h'(a)|&\lesssim a^{-1}+ a^{p-1} \,  \quad\text{for all }  a> 0.
     \end{split}
    \end{equation}
\item \textbf{Local well--posedness, non--zero average dispersion:}\label{ass: dav not zero}
$h$ is  continuous on $[0,\infty)$ and continuously differentiable on $(0,\infty)$ with $\lim_{a\to 0}h'(a)a=0$. There exist increasing functions
$J_1, J_2:[0,\infty)\to [0,\infty)$ such that
\begin{equation} \label{eq:ass dav not zero}
  \begin{split}
	|h(a)|&\le J_1(a) \,  \quad\text{ for all }  a\ge 0, \\
	|h'(a)|&\le J_2(a)(1+a^{-1}) \,  \quad\text{ for all }  a> 0.
  \end{split}
\end{equation}
\end{assumptions}
\setcounter{saveassumptions}{\value{assumptions}}
For global well--posedness, when the average dispersion is non--zero,
we need to assume in addition
\begin{assumptions}
\setcounter{assumptions}{\value{saveassumptions}}
\item \textbf{Global well--posedness, non--zero average dispersion:}\label{ass: dav not zero global}
 For an increasing function  $\wti{J}:[0,\infty)\to [0,\infty)$ with
\begin{equation}\label{eq: growth condition wtiJ}
	\lim_{a\to\infty} \frac{\wti{J}(a)}{a^4} = 0
\end{equation}
and for some $0\le p\le 4$ the nonlinearity $h$ satisfies
\begin{equation}
  \begin{split}
	& h(a)\le \wti{J}(a)(1+a^p) \,  \text{ for all }  a\ge 0, \text{ when }\dav >0,  \\
	& h(a)\ge  -\wti{J}(a)(1+a^p) \, \text{ for all }  a\ge 0,  \text{ when }\dav <0.
  \end{split}
\end{equation}
\end{assumptions}
\setcounter{saveassumptions}{\value{assumptions}}
Above, we use the convention $f\lesssim g$, if there exists a finite constant $C>0$ such that $f \le Cg$.

\begin{remarks}
 \begin{theoremlist}
\item
  In assumption \ref{ass: dav zero}, the growth condition
  \eqref{eq:ass dav zero} on $h$ is consistent
  with the fact that the nonlocal nonlinearity in \eqref{eq:main}
  is bounded on $L^2(\R)$ for $P(z)=|z|^pz$ and $\psi\in L^\frac{4}{4-p}$
  only for $0\le p\le 4$, see Lemma \ref{lem:L^2 bound}.
  However, the assumption on $h'$ in \eqref{eq:ass dav zero} is
  rather weak, allowing a blow--up of $h'$ for small $a$.
  For example, our assumptions cover even highly oscillating
  nonlinearities of the form
  \begin{equation}\label{eq:oscillating example}
  	h(a)= a^{\delta}\sin\left(  \frac{1}{a^{\kappa}}\right)
  \end{equation}
  with $h(0)=0$ and $0<\kappa<\delta\le 4$.
  Also sign changing
  polarizations of the form
  \begin{equation}\label{eq:sign changing example}
  	P(z) = c_1|z|^{p_1}z - c_2|z|^{p_2}z
  \end{equation}
  for $c_1,c_2>0$ and exponents $0\le p_1\le p_2\le 4$ are covered by assumption \ref{ass: dav zero}.
\item
  Assumption \ref{ass: dav not zero} is even weaker:
  $h$ only has to be locally bounded, without any growth condition
  at infinity and the possibility of large oscillations of $h'(a)$ for
  small and large values of $a$.
  The example \eqref{eq:oscillating example} satisfies assumption
  \ref{ass: dav not zero} for all $0<\kappa<\delta$ and also assumption
  \ref{ass: dav not zero global} for all $0<\kappa<\delta <8$.
  The example \eqref{eq:sign changing example} satisfies assumption
  \ref{ass: dav not zero}  for all $0\le p_1\le p_2<\infty$ and
  \ref{ass: dav not zero global} for all $0\le p_1\le p_2<8$.
  Polarizations given by the power law $P(z)=|z|^pz$ satisfy assumption \ref{ass: dav not zero} for all $p \ge 0$.
  Since $h(a)= a^p$ is bounded from below in this case,
  assumption \ref{ass: dav not zero global} is also satisfied for all
  $p>0$ when $\dav<0$.
  However, to satisfy assumption \ref{ass: dav not zero global}, when
  $\dav>0$, we need to restrict to $p<8$ for power law polarizations.
 \item
  Moreover, assumptions \ref{ass: dav zero}, \ref{ass: dav not zero}, and \ref{ass: dav not zero global} allow for saturated nonlinearities.
  For example, the polarization is allowed to be of the form
  \begin{equation*}
  	P(z)=\f{|z|^2z}{1+\sigma |z|^2}\, ,
  \end{equation*}
  with $\sigma>0$. In this case, we have $h(a)= a^2/(1+\sigma a^2)$, i.e., $h$ is  asymptotically constant for
  large values of $a$.
\end{theoremlist}
\end{remarks}

Before presenting our main results, we make the notion of a solution more precise, see \cite{Cazenave,Tao}:
Let $X_1,X_2$ be Banach spaces.
A function $u: \R \times  [-M_-,M_+]\to \C$, for some positive $M_\pm$, is called a (local) strong solution of \eqref{eq:main} if  $u\in \calC([-M_-, M_+], X_1)\cap \calC^{1}((-M_-, M_+), X_2)$ satisfies the equation
\begin{equation}\label{eq:strong-solution-1}
	i\partial_t u+ \dav \partial_x^2u+ Q(u)=0 \notag
\end{equation}
and $u(\cdot,0)= u_0$, where the
nonlocal nonlinearity $Q$ is given by
\beq
 	Q(u(t)):=\int_\R T_r^{-1}(P(T_r u(t)) )\psi(r)dr. \notag
\eeq

If $\dav\not =0$, we take $X_1=H^1(\R)$ and $X_2= H^{-1}(\R)$,
for a definition of the scale of Sobolev spaces $H^s(\R)$ see the next section. This is motivated by the fact that under suitable conditions on the nonlinearity, see Lemma \ref{lem:nonlinear H^1 bound},  $Q$ maps $H^1(\R)$ into itself and thus, if $u(t)\in H^1(\R)$ solves \eqref{eq:main}, then $\partial_t u(t)\in H^{-1}(\R)$.
If $\dav=0$, then we take $X_1=X_2=L^2(\R)$, since, under suitable conditions on the nonlinearity, $Q$ maps $L^2(\R)$ into itself, see Lemma \ref{lem:L^2 bound}.

It is well--known that $u$ is
a strong solution of \eqref{eq:main} with initial datum $u_0$ if
and only if $u\in \calC([-M_-, M_+], H^1(\R))$ for some positive
$M_\pm$ and $u$ fulfills  the Duhamel formula
  \begin{equation}\label{eq: Duhamel u}
  u(t) = e^{it \dav \partial _x^2}u_0 +i \int _0 ^t e^{i(t-t')\dav \partial _x^2} Q(u(t'))\, \mathrm{d}t'
  \end{equation}
for all $t\in [-M_-, M_+]$, see \cite[Proposition 3.1.3]{Cazenave} and, also, \cite{Ball,Tao}.
It is a global strong solution, if $[M_-,M_+]$ can be replaced
by $\R$.
In the following, we will mainly work
with the integral version \eqref{eq: Duhamel u}
instead of \eqref{eq:main}.

The Cauchy problem \eqref{eq:main}, or better the integral equation
\eqref{eq: Duhamel u}, is \emph{locally well--posed} in $H^1(\R)$
for $\dav \neq 0$ if, for any
initial data $u_0\in H^1(\R)$, there exist times $M_\pm>0$
and a ball $B$ in $H^1(\R)$ containing $u_0$ that for each
$\phi\in B$ there exists a unique  strong solution
$u\in \calC([-M_-,M_+], H^1(\R))$ of \eqref{eq: Duhamel u} with initial datum $\phi$
and the map $\phi\mapsto u$ is continuous from $B$ to $\calC([-M_-,M_+], H^1(\R))$.
It is \emph{globally well--posed} in $H^1(\R)$ if we
can take $M_\pm$ arbitrary large.
For $\dav=0$, we replace $H^1(\R)$ by $L^2(\R)$.

\medskip
For zero average dispersion, we get global
well--posedness just assuming \ref{ass: dav zero}.
\begin{theorem}[Global well--posedness in $L^2(\R)$ for $\dav=0$]\label{thm:well-posedness in L^2}
Let  $h$ satisfy assumption \ref{ass: dav zero} and
$\psi\in L^1 (\R) \cap L^{\frac{4}{4-p}}(\R)$.
Then the Cauchy problem \eqref{eq:main} is globally well--posed
in $ L^2(\R)$
and the mass is conserved.
\end{theorem}

For non--zero average dispersion we get local well--posedness under assumption \ref{ass: dav not zero}.
\begin{theorem}[Local well--posedness in $H^1(\R)$ for $\dav\not=0$] \label{thm:local well-posedness in H^1}
 Let $h$ satisfy assumption
 \ref{ass: dav not zero} and $\psi\in L^1(\R)$.
 Then  the Cauchy problem \eqref{eq:main} is locally
 well--posed in $H^1(\R)$
 and the mass and energy are conserved.
\end{theorem}

When the average dispersion is not zero we need assumptions
\ref{ass: dav not zero} and \ref{ass: dav not zero global}
for global existence.
\begin{theorem}[Global well--posedness in $H^1(\R)$ for $\dav\not=0$] \label{thm:global well-posedness in H^1}
 Let $h$ satisfy assumptions
 \ref{ass: dav not zero}, \ref{ass: dav not zero global}
 and $\psi\in L^1(\R)\cap L^{\frac{4}{4-p}}(\R)$.
 Then  the Cauchy problem \eqref{eq:main} is globally
 well--posed in $H^1(\R)$
 and the mass and energy are conserved.
\end{theorem}
The mass, which in nonlinear optics is the power of the pulse, is given by
\begin{align} \label{mass}
	m(u(t))\coloneqq \|u(t)\|_{L^2}^2\, .
\end{align}
The energy is given by
\begin{align} \label{energy}
	E(u(t))\coloneqq \f{\dav}{2} \|\partial_x u (t)\|_{L^2}^2-\iint_{\R^2} V(|T_r u(t)|)\,dx\psi(r)dr\, ,
\end{align}
where $V(a)=\int_0^a P(s)\, \mathrm{d}s=\int_0^a h(s)s\, \mathrm{d}s$ for $a\ge0$.

\begin{remark}
  For zero--average dispersion, any solution of
  \eqref{eq: Duhamel u} in the space $\calC(\R, L^2(\R))$
  is continuously differentiable, i.e., in
  $\calC^1(\R, L^2(\R))$. Similarly, for $\dav\neq 0$,
  any solution of   \eqref{eq: Duhamel u} in
  $ \calC(\R, H^1(\R))$ is in the space
  $\calC(\R, H^1(\R))\, \cap \, \calC^1(\R, H^{-1}(\R))$.
  Thus our results also show what is known as
  \emph{unconditional uniqueness} in
  the literature.
  For a Kerr--type polarization $P(z)= |z|^2z$, global
  well--posedness was
  proven in \cite{AK, ZGJT01}. The assumptions of
  \cite{AK} on the local dispersion profile imply that
  $\psi\in L^\infty(\R)$ and has compact support.
  Our Theorems \ref{thm:well-posedness in L^2},
  \ref{thm:local well-posedness in H^1},
  and \ref{thm:global well-posedness in H^1} show that
  local as well as global well--posedness still holds for a much larger class of
  nonlinearities with minimal smoothness assumptions on the nonlinearity.
  We can also allow for a larger class of dispersion profiles in the
  dispersion managed NLS.
 \end{remark}
  While we need assumption \ref{ass: dav not zero global} in the
  proof of Theorem \ref{thm:global well-posedness in H^1},
  we also have a global well--posedness result for \emph{small initial data} just assuming   \ref{ass: dav not zero}.
\begin{theorem}[Small data global well--posedness]\label{thm: small data global well-posedness}
	Let $\dav\not=0$ and  $h$ satisfy assumption
	\ref{ass: dav not zero}.
	\begin{theoremlist}
		\item For any initial datum
			$u_0\in H^1(\R)$ with small enough $H^1$-norm,
			 the Cauchy problem \eqref{eq:main} is globally well--posed
			 when $\psi\in L^1(\R)$.
		\item If $J_1(a) \lesssim 1+a^8 $ for $a\ge 0$, then
			 the Cauchy problem \eqref{eq:main} is globally well--posed
			 for initial conditions
  			$u_0\in H^1(\R)$ with $\|u_0\|_{L^2}$ small enough when
  			$\psi\in L^\infty(\R)\cap L^1(\R)$.
		\item If $\lim_{a\to 0}J_1(a)/a^4=0$ then
			 the Cauchy problem \eqref{eq:main} is globally well--posed
			 for initial conditions
  			$u_0\in H^1(\R)$ with $\|u_0'\|$ small enough
  			{\rm(}depending on $\|u_0\|${\rm)} when  $\psi\in L^1(\R)$.
	\end{theoremlist}
\end{theorem}

\begin{remark}
	We note that in $L^2(\R)$ there are initial conditions $u_0\in H^1(\R)$ with
	$\|u_0\|_{L^2}$ large and $\|u_0'\|_{L^2}$ arbitrarily small.
	An example is given by
	$u_\delta(x)=C(3\delta/2)^{1/2}(1-\delta|x|)_+$ for small
	$\delta>0$ and $C>0$. Then $\|u_\delta\|_{L^2}=C$ and
	$\|u_\delta'\|_{L^2} = \sqrt{3}C\delta\to 0$ as $\delta\to 0$.
	Thus the last statement in  Theorem
	\ref{thm: small data global well-posedness} is not empty.
\end{remark}

For orbital stability, we consider ground states of \eqref{eq:main},
that is,
stationary, or standing wave, solutions of \eqref{eq:main} in the form
$u(x,t)= e^{-i\omega t}f(x)$ with minimal energy.
These are given by  minimizers of the  nonlocal nonlinear
constrained variational problem
\beq\label{constrained variational problem}
E_\lambda^{\dav}=\inf \{ E(f):\; f\in X, \|f\|^2_{L^2}=\lambda\} \notag
\eeq
where $\lambda>0$, $\dav\geq 0$, and $X=L^2(\R)$ for $\dav=0$; $X=H^1(\R)$
for $\dav>0$.
When $\dav<0$ the `kinetic energy' term in the energy is
non--positive, so one should maximize
the `energy' when $\dav<0$.
Equivalently, one could keep the sign of $\dav$ positive
and flip the sign of the nonlinearity. Thus the case $\dav<0$
corresponds to `defocusing nonlinearities', where, at least
for Kerr--type nonlinearities one does not expect to have
stationary solutions, see Remark \ref{rem:defocussing case}.
Thus, for orbital stability, we only consider $\dav\ge 0$.
If $\dav> 0$, assumption \ref{ass: dav not zero global} implies
that the energy is coercive, see \eqref{eq:coercivity-2}.
Every nonlinear ground state $f$ weekly solves the equation
\beq\label{eq: stationary problem}
\omega f =-\dav f'' -\int _\R T_r ^{-1} (P(T_r f))\psi (r)dr  \notag
\eeq
for some Lagrange multiplier $\omega$.

We denote by $S_\lambda^{\dav}$ the set of all ground states
\begin{equation}\label{eq:ground state set}
  S_\lambda^{\dav}=\{f\in X: \; E(f)=E_{\lambda}^{\dav}, \|f\|^2_{L^2}
  =\lambda \},
\end{equation}
for $\lambda>0$ and $\dav\geq0$.
To get orbital stability of $S^\dav_\lambda$, even
to guarantee that   $S^\dav_\lambda\not=\emptyset$,
see Theorem \ref{thm:existence}, we need additional assumptions.
We distinguish between saturated and non--saturated
nonlinearites.

\smallskip
\begin{assumptions}
\setcounter{assumptions}{\value{saveassumptions}}
\item\label{ass:A4} \textbf{Non--saturated nonlinearity:} There exists a constant $p_0>2$ with
	\begin{equation}\label{eq:strong A-R}
		h(a)a^2 \ge p_0 \int_0^a h(s)s\, \mathrm{d}s \text{  for all } a>0\, .
	\end{equation}
\item\label{ass:A5}  \textbf{Saturated nonlinearity:}
	There exists a continuous function
	$p:[0,\infty)\to (2,\infty)$, where we allow
	$\lim_{a\to\infty}p(a)=2$, such that
		\begin{equation}\label{eq:saturating A-R}
			h(a)a^2 \ge p(a) \int_0^a h(s)s\, \mathrm{d}s\text{ for all } a>0.
		\end{equation}
\item\label{ass:A6} There exists $a_0>0$ with
	$\int_0^{a_0} h(s)s\, \mathrm{d}s>0$.
\end{assumptions}

\begin{remarks}
  \begin{theoremlist}
    \item
    	In terms of
		$V(a) =\int_0^a P(s)\, \mathrm{d}s = \int_0^a h(s)s\, \mathrm{d}s$
		the condition \eqref{eq:strong A-R} is equivalent to
		$V'(a)a\ge p_0 V(a)$ for $a> 0$. This is the well-known
		Ambrosetti--Rabinowitz condition from the calculous of
		variations \cite{AR}.  Condition \eqref{eq:saturating A-R}
		is a weakened version of the classical Ambrosetti--Rabinowitz
		condition which allows for \emph{saturating nonlinearities}.
		That the  variational approach for constructing nonlinear
		ground states also works under such a weaker condition is
		less known, see \cite{HLRZ} for the case of dispersion
		management solitons.

    	Condition \ref{ass:A6} together with \ref{ass:A5} (or \ref{ass:A4}) guarantees that the nonlinearity
	is positive for large $a$. Indeed, let
	$V(a)= \int_0^{a} h(s)s\, \mathrm{d}s$.
	Since $V$ is continuous and $V(a_0)>0$, there exists
	$a_0<d\le \infty$ such that $V(a)>0$ for $a_0\le a<d$ and $V(d)=0$
	if $d<\infty$. The assumption \ref{ass:A5} is equivalent to
	\begin{align*}
		\frac{V'(a)}{V(a)} \ge p(a) a^{-1} \quad \text{ for } a_0<a<d
	\end{align*}
	and integrating this from $a_0$ to $a$ shows
    \begin{align}\label{eq:quadratic lower bound for V}
		V(a) \ge V(a_0) \exp\left( \int_{a_0}^a p(s)\, \frac{\mathrm{d}s}{s} \right)	
			> V(a_0) \left(\frac{a}{a_0}\right)^2
  	\end{align}
  	since $p(a)>2$. Thus $V$ grows at least quadratically near
  	infinity, hence $d=\infty$. Moreover, \eqref{eq:saturating A-R}
  	shows
  	\begin{align*}
  	h(a)a^2 \ge p(a) V(a) >2 V(a_0)\left(\frac{a}{a_0}\right)^2
  	\end{align*}
  	so $h(a)> 2 V(a_0) a_0^{-2}>0$ for all $a\ge a_0$.

  Note that when  $\lim_{a\to\infty} p(a)=2$ the condition
  \ref{ass:A5} allows the nonlinearity
  to \emph{saturate}.
  The bound \eqref{eq:quadratic lower bound for V} is the best one
  can get, since $V(a)$ will asymptotically quadratically when
  $p(a)\to 2$ fast enough as $a\to\infty$. In this case,
  the nonlinear polarization $P(z)=h(|z|)z$ is asymptotically linear
  for $|z|$ large.

  Under the stronger condition \ref{ass:A4}, which does not allow
  for saturation, one has the lower bound
  \begin{align*}
  	V(a) \ge V(a_0) \left(\frac{a}{a_0}\right)^{p_0}
  \end{align*}
  for all $a\ge a_0$. In this case $h(a)\ge  V(a_0) a_0^{-p_0} a^{p_0-2}$
  for $a\ge a_0$, so $h$ grows at least with exponent $p_0-2$ when the nonlinearity does not saturate.

	\item If one prefers to have a local condition on the nonlinearity $h$,
	a suitable substitute for \eqref{eq:strong A-R} is
		\begin{equation}\label{eq:strong A-R-local}
		h'(a)a \ge (p_0-2) h(a) \text{  for all } a>0\,
	\end{equation}
	and
		\begin{equation}\label{eq:saturating A-R-local}
			h'(a)a \ge (p(a)-2) h(a)\, \text{ for all } a>0
		\end{equation}
	for \eqref{eq:saturating A-R}.
	In fact, \eqref{eq:strong A-R-local} is equivalent to
	$(h(a)a^2)' \ge p_0h(a)a$, and integrating this, one gets
	\eqref{eq:strong A-R}. Similarly \eqref{eq:saturating A-R-local}
	implies \eqref{eq:saturating A-R}, replacing $p(a)$ with
	$\inf_{0<s\le a} p(s)$.
  \end{theoremlist}
\end{remarks}

\medskip
To state the last theorem, we need one more notation: Given $r\ge 1$ we say that
$\psi\in L^{r+}$ if $\psi\in L^{r+\delta}$ for some $\delta>0$.
We also set $L^{\infty+} = L^\infty$.

\begin{theorem}\label{thm:Stability result}
  Suppose that the nonlinearity $h$ satisfies assumption \ref{ass:A6} and
  either of the following:
  \begin{theoremlist}
	\item \textbf{Zero average dispersion, non--saturated nonlinearity:}
		The nonlinearity $h$ satisfies assumption  \ref{ass: dav zero}
		for its derivative $h'$, except that the growth condition on $h$
		is slightly strengthened to $|h(a)|\lesssim a^{p_1}+ a^{p_2}$
		for some $0< p_1\le  p_2<4$.
		It also satisfies assumption \ref{ass:A4}, and  the density
		$\psi\in L^{\frac{4}{4-p_2}+}(\R)$ has compact support.
	\item \textbf{Zero average dispersion, saturated nonlinearity:}
		The nonlinearity $h$ satisfies assumption  \ref{ass: dav zero}
		for its derivative $h'$, except that the growth condition on $h$
		is strengthened to $|h(a)|\lesssim a^{p_1}+ a^{p_2}$
		for some $1\le  p_1\le  p_2< 3$.
		It also satisfies assumption \ref{ass:A5}, and  the density
		$\psi\in L^{\frac{4}{3-p_2}+}(\R)$ has compact support.
	\item  \textbf{Positive average dispersion, saturated and non--saturated nonlinearities:}
		The nonlinearity $h$ satisfies assumption
		\ref{ass: dav not zero}  and  the bound
		$|h(a)|\lesssim a^{p_1}+ a^{p_2}$ for some
		$0<p_1\le p_2<8$.
		The density  $\psi\in L^{\frac{4}{8-p_2}+}(\R)$
	    has compact support.
	    Moreover, $h$ satisfies either assumption
		\ref{ass:A4} or assumption \ref{ass:A5}.
  \end{theoremlist}
Then there exists a  critical threshold $0\le \lambda_{cr}^{\dav}<\infty$
such that
if $\lambda>\lambda_{cr}^{\dav}$ then $S_\lambda^{\dav}\neq \emptyset$.
Moreover, the set of ground states  $S_\lambda^{\dav}$
is orbitally stable in the sense that for every
$\veps >0$, there exists  $\delta>0$ such that if $u_0\in X$ with
\begin{equation*}
 \inf_{f\in S_\lambda^{\dav}}\|u_0-f\|_{X} < \delta,
\end{equation*}
then the solution $u$ with the initial datum $u_0$ satisfies
$$
\inf_{f\in S_\lambda^{\dav}}\|u(\cdot , t)-f\|_{X} < \veps
$$
for all $t\in \R$, where $X= H^1(\R)$ if $\dav>0$ and
$X= L^2(\R)$ if $\dav=0$.

In addition, if $\dav>0$ and
$0<\lambda<\lambda_{\text{cr}}^{\dav}$ then
$S_\lambda^{\dav}= \emptyset$.
\end{theorem}
\begin{remarks}\label{rem:defocussing case}
  \begin{theoremlist}
  	\item In the third assumption of the above theorem, the
  		condition $|h(a)|\lesssim a^{p_1} + a^{p_2}$ for
  		some $0<p_1\le p_2<8$ and all $a>0$ clearly
  		implies assumption \ref{ass: dav not zero global}.
  	\item For the existence of a critical threshold $\lambda_{cr}$ for which the set of ground states $S_\lambda^{\dav}$ is not empty when $\lambda>\lambda_{cr}$ slightly weaker assumptions suffice, see Theorem \ref{thm:existence}.
  	      For saturating nonlinearitioes we need to restrict
  	      the range of $p$ from $0<p<4$ to $1\le p<3$ in order to
  	      have $S_\lambda^{0}\neq\emptyset$.
  		 The additional assumptions are needed to ensure global
  		 existence of solutions.
  	\item 	
	If the average dispersion is negative, $\dav<0$, the
	nonlinearity in \eqref{eq:main} is \emph{defocusing}, at least when it is given by the Kerr approximation. For the local NLS it
	is known that there are no stationary solutions, i.e., solitons, in this case.
	For the dispersion managed NLS this is not known. While there are some
	numerical simulations, which show stable propagation of pulses for
	negative average dispersion $\dav<0$ with $|\dav|$ small, it seems
	that these pulses lose energy over time by radiation. Thus they are
	not expected to be true stationary solutions, see
	\cite[Remark 3.2]{ZGJT01}.
	However, this has not been shown rigorously.
  \end{theoremlist}
\end{remarks}

\subsection{The connection to nonlinear optics}
\label{sec:connection}
Equation \eqref{eq:main} is an averaged version of the local, but non--autonomous
dispersion managed NLS
 \begin{equation} \label{eq:NLS-gen}
      i \partial_t w = -d_{\text{loc}}(t) \partial_x^2 w - P(w),
 \end{equation}
where the dispersion $d_{\loc}(t)$ is parametrically modulated and
$P$ is the nonlinear interaction due to the polarizability of the glass--fiber cable. The constant $\dav$ is the average dispersion over one period along the cable and the function $\psi$ is the density of a probability measure related
to the mean--zero periodic part of the local dispersion profile,
\begin{equation}
	d_{\text{loc}}(t)=\dav +d_{\text{per}}(t)\, .
\end{equation}
 Local well--posedness of the non--averaged equation
\eqref{eq:NLS-gen} has been shown for power--law type
nonlinearities   in  \cite{ASS, CKL,CLA}, for example.
The question of global existence
versus finite time blowup of solutions of \eqref{eq:NLS-gen}
has been investigated in \cite{ASS}, see also \cite{MH2} for
related results.

In the case of strong dispersion management, one assumes that
the mean zero periodic part $d_{\text{per}}$  is give by
\begin{equation*}
	 d_{\text{per}}(t)= \veps^{-1} d_0(t/\veps)
\end{equation*}
with $d_0$ periodic, of period $L>0$ and zero mean, and
$\veps>0$ small. Since \eqref{eq:NLS-gen} is non--autonomous with a highly oscillating periodic local dispersion, Gabitov and Turitsyn  \cite{GT96a,GT96b} found an approximation which is good for small $\veps$, i.e., in the regime of strong dispersion management.
Roughly, the idea is as follows: Let $T_r= e^{ir \partial_x^2}$,
$D(t)=\int_0^t d_0(s)\, \mathrm{d}s$, and make the ansatz
\begin{equation}\label{eq:change variables}
	w(x,t) = T_{D(t/\veps)}v(\cdot,t)(x)\, .
\end{equation}
 Then \eqref{eq:NLS-gen} is equivalent to
 \begin{equation}\label{eq:NLS-gen-2}
 	i\partial_t v= -d_{\text{av}}\partial_x^2 v - T_{D(t/\veps)}^{-1}\big[P(T_{D(t/\veps)} v)\big]
 \end{equation}
 which now contains the fast oscillating term $T_{D(t/\veps)}$ in the nonlinearity, but the linear part is constant in $t$;  since $d_0$ has mean zero and period $L$, the cumulative dispersion $D(t/\veps)$ is periodic with period $\veps L$.
 The idea of Gabitov and Turitsyn, for the special case of a Kerr nonlinearity, is to average the fast oscillating nonlinear terms containing $T_{D(t/\veps)}$ over one period in $t$, which yields the dispersion managed NLS
 \begin{equation}
 	\begin{split}\label{eq:GT-1}
        i\partial_t u
        &= - d_{\text{av}} \partial_x^2  u
            - \frac{1}{\veps L}\int_{0}^{\veps L}
                T_{D(s/\veps)}^{-1}\big[P(T_{D(s/\veps)} u)\big]
              \, \mathrm{d}s \\
        &= - d_{\text{av}} \partial_x^2  u
            - \frac{1}{L}\int_{0}^{L}
                T_{D(s)}^{-1}\big[P(T_{D(s)} u)\big]
              \, \mathrm{d}s
 	\end{split}
 \end{equation}
 where $u$ now is the average profile of the pulse $v$. This is analogous to Kapitza's treatment of the unstable pendulum, which is stabilized by fast
oscillations of the pivot, see \cite{LandauLifshitz}.  Proofs of the averaging theorem were given in \cite{ZGJT01} and \cite{CKL} for Kerr type nonlinearity and \cite{CLA} for power--law type nonlinearities.
In the case of fast dispersion management, where the periodic mean--zero
part of the dispersion profile is given by
$d_{\text{per}}(t)=d_0(t/\veps)$, an averaging theorem is
proven in \cite{ASS}.

We prefer to rewrite \eqref{eq:GT-1} a bit: Introduce a probability measure $\mu$ on the Borel sets of $\R$ by
$\mu(B)\coloneqq \frac{1}{L}\int_0^L \id_B(D(s))\, \mathrm{d}s$ and make to change of variables $r=D(s)$ to see that
\eqref{eq:GT-1} is equivalent to
  \begin{equation*}
     i\partial_t u
         = - d_{\text{av}} \partial_x^2  u
            - \int_\R
                T_{r}^{-1}\big[P(T_{r} u)\big]
              \, \mu(dr)
 \end{equation*}
 which is equivalent to \eqref{eq:main} when $\mu$ has density $\psi$.

 Note that since the local mean zero periodic dispersion profile  $d_0$ is locally integrable, its integrated version $D$ is bounded, hence the probability measure $\mu$ has compact support. In particular, its density, once it exists, has compact support in all physically interesting cases. The existence and suitable $L^p$ properties
of the density $\psi$ follow from physically natural conditions on the local mean zero periodic dispersion profile $d_0$.

The model case, which is usually assumed, is a two step local dispersion profile
$d_{0}=d_{\text{model}}$ with
$d_{\text{model}}(t)=+1$ if $0\le t\le 1$ and
$d_{\text{model}}(t)=-1$ if $1<t<2$, extended periodically to
$t\in\R$. For such a model case the probability density $\psi$ is given by
\begin{equation*}
\psi_{\text{model}}=\id_{[0,1]}\, ,
\end{equation*}
the characteristic function of the interval $[0,1]$.
This simplifying assumption is often made but we will
not make it here.
We refer to \cite[Section 1.2]{HuLee2012} or
\cite[Section 1.2]{ChoiHuLee2016} for a detailed discussion
how the probability density $\psi$ is connected to the local periodic
dispersion profile, see \cite[Lemma 1.4]{HuLee2012}.
Most important for us is the criterion that, if $d_{0}$
stays away from zero and changes its sign
finitely many times  over one period, then
\begin{equation}
	\psi\in L^q \text{ for } q>1 \text{ whenever } \int_0^L |d_{\text{per}}(s)|^{1-q}\, \mathrm{d}s <\infty \, ,
\end{equation}
see \cite[Lemma 1.4]{HuLee2012}.
In particular, all the $L^q$--type conditions on $\psi$
are fulfilled for all physically reasonable local
dispersion profiles $d_{\text{per}}$.

\begin{remark}\label{rem:focusing-defocusing}
  The well--known local NLS in one space dimension is often written in the form
  \begin{align}\label{eq:normal-NLS}
  	i\partial_t u
  		= -\partial_x^2 u -\lambda P(u)
  \end{align}
  and a coupling constant
  $\lambda\in\R$.
In this case $\lambda>0$ is called a focusing and $\lambda<0$ is called a defocusing nonlinearity.
  Thus for the dispersion managed NLS \eqref{eq:main}, $\dav>0$ corresponds to the focusing, and $\dav<0$ to the defocusing, case of the usual local NLS, at least when $h$ is nonnegative, where the nonlinearity is given by $P(u)= h(|u|)u$.
\end{remark}

\bigskip
This paper is organized as follows. In Section \ref{sec:Nonlinear estimates} we gather the necessary nonlinear bounds. Due to the
nonlocality of the nonlinearity, these are quite different from what is usually
used in the study of NLS. Local existence is done in Section
\ref{sec: local existence  for DMNLS}. Since our assumptions on the
nonlinearity are rather weak, the existence proof does not immediately yield
continuous dependence on the initial data, at least when $\dav\neq 0$.
This local well--posedness is done in Section \ref{sec: local well--posedness}.
Global existence  is  based on mass and
energy conservation.
Due to our rather weak differentiability assumptions on the nonlinearity,
the usual approach to prove conservation of energy and mass is not
applicable in our case, see the discussion in the beginning of Section
\ref{sec: energy conservation}. We use a twisting trick to avoid the
usual approximation arguments. Our argument directly proves
differentiability of the mass and energy
and allows for low regularity solutions.
The proof of global existence is finished in Section \ref{sec: global existence}
and in Section \ref{sec:Stability result} we give the proof of orbital
stability of the set of ground states for non--negative average dispersion.

\section{Nonlinear estimates} \label{sec:Nonlinear estimates}

Before we collect the estimates we need, let us introduce some notations.
$L^p(\R)$ for $1\leq p \leq \infty$ and $H^s(\R)$, $s\in\R $, are the usual Lebesgue and Sobolev spaces with norms
$\| \cdot \| _{L^p}$ and $\| \cdot \| _{H^s}$, respectively.
That is,
$L^p(\R)$ is the space of (equivalence classes of) functions $f$ for which
\begin{equation}
	\|f\|_{L^p}= \left( \int_{\R } |f(x)|^p\, dx  \right)^{1/p} <\infty\, .  \notag
\end{equation}
For $f\in L^2(\R)$, we will simply write $\|f\|_{L^2}=\|f\|$. The Sobolev space is given by \
\begin{equation}
	H^s(\R)= \left\{ f\in\calS^*:\, \int_\R \la \eta\ra^{2s}|\hatt{f}(\eta)|^2\, d\eta <\infty  \right\}  \notag
\end{equation}
with the norm $\|f\|_{H^s}=\|\la \cdot\ra^s\hatt{f}\|$,
where $\calS^*=\calS^*(\R)$ denotes the tempered distributions on $\R $,
$\la \eta \ra\coloneqq (1+\eta^2)^{1/2}$, and
$\hatt{f}$ is the Fourier transform of $f$, defined by
\begin{equation}
	\hatt{f}(\eta)\coloneqq (2\pi)^{-1/2}\int_\R e^{-ix\eta}f(x)\, dx  \notag
\end{equation}
for $f\in \calS$, the Schwartz space of infinitely smooth, rapidly decreasing functions, and extended by duality to the space of tempered distributions $\calS^*$.

We denote by $L_t^q(J, L_x^p(I))$, for $1 \le p, q < \infty$ and intervals $I, J$, the space of all functions $u$ for which
$$
\|u\|_{L_t^q(J, L_x^p(I))}=\left(\int _J \left(\int _I |u(x,t)|^p dx\right)^{\f{q}{p}}dt\right)^{\f{1}{q}}
$$
is finite. If $p=\infty$ or $q=\infty$, use the essential
supremum instead. For notational simplicity, we write
$L^q( L^p)$ for $L_t^q(\R, L_x^p(\R))$.
For a Banach space $X$ with norm $\|\cdot\|_X$ and an interval
$J$, $\calC(J, X)$ is the space of all continuous functions
$u:J \to X$. When $J$ is compact, it is a Banach space with norm
\begin{equation*}
	\|u\|_{\calC(J, X)}=\sup _{t\in J}\|u(t)\|_{X}
\end{equation*}
and $\calC^1(J, X)$ is the Banach space of all continuously differentiable functions $u:J \to X$.

Now we gather some properties of the solution operator $T_r=e^{ir \partial_x^2}$ for the free Schr\"{o}dinger equation in spatial dimension one. It is a unitary operator on $L^2(\R)$ and, also, on $H^1(\R)$ and therefore for every $r\in \R$
$$
\|T_r f\|=\|f\| \quad \mbox{and} \quad \|T_r f\|_{H^1}=\|f\|_{H^1}.
$$

The following is the one-dimensional Strichartz estimate in the form that we need.

\begin{lemma}[One-dimensional Strichartz estimates]\label{lem:Strichartz}
\begin{theoremlist}
\item
Let $2\leq p \leq \infty$
so that
$$
\frac{1}{p}+\frac{2}{q}=\frac{1}{2}.
$$
If $f\in L^2(\R)$,  then the map $r \mapsto T_r f$ belongs to $L^q(L^p)\cap \calC(\R, L^2)$ and
\beq \label{est:strichartz ineq}
\|T_r f\|_{L^q(L^p)}\lesssim \|f\|,
\eeq
where the implicit constant depends only on $p$.
Moreover, if $f\in H^1(\R)$, then  the map $r \mapsto T_r f$ is in $\calC(\R, H^1)$.

\item Let $J$ be a bounded interval containing zero.  If $F\in L^1(J, L^2)$, then the map
$$
r \mapsto \Psi_F(r) :=\int_0 ^r T_{(r-r')} F(\cdot, r')dr'
$$
belongs to $L^\infty(J, L^2)\cap \calC(J, L^2)$
and
\bdm
\left\|\Psi_F \right\|_{L^\infty(J, L^2)}\lesssim \|F\|_{L^1(J, L^2)}.
\edm
Moreover, if $F\in L^1(J, H^1)$, then $\Psi_F$ is in  $\calC(J, H^1)$.
\end{theoremlist}
\end{lemma}
The Strichartz inequalities have a long history. The first proof by Strichartz \cite{Strichartz}, valid in all dimensions, was for the special case $p=q$. It was then later extended by several authors, see, for example, \cite{GV,KT}.
The above formulation is from \cite{Cazenave} for the case of one space dimension.

To state the space time bounds we need, which are based on
Strichartz type estimates,
we introduce one more notation. For a non--negative function
$\psi$ on $\R$ we denote by $L^q(\R^2, dx \psi dr)$, $1\leq q < \infty$, the Banach space of all functions with the weighted norm
\begin{equation*}
  \|u\|_{L^q(\R^2, dx \psi dr)}= \left(\iint _{\R^2} |u(x,r)|^q dx \psi(r) dr\right)^{1/q}.
 \end{equation*}

\begin{lemma}\label{lem:L^2boundedness}
Let $2\le q\le 6$ and $\psi\in  L^{\f{4}{6-q}}(\R)$. Then for all $f\in L^2(\R)$,
\beq\label{eq:L^2 boundedness}
	\|T_r f\|^q_{L^q(\R^2,dx\psi dr)}
	\lesssim
		\|f\|^q,
\eeq
where the implicit constant depends only on the $L^{\f{4}{6-q}}$ norm of $\psi$.
\end{lemma}
\begin{proof}
The bound \eqref{eq:L^2 boundedness} is exactly the same as
provided by Lemma 2.1 in \cite{ChoiHuLee2016}.  We give a
simpler proof.
Using H\"older's inequality with exponents $\frac{4}{q-2}$ and $\frac{4}{6-q}$ in the $r$-integral and then Strichartz inequality from Lemma \ref{lem:Strichartz} one obtains
\beq
	\iint_{\R^2} |T_rf|^{q}\, dx \psi(r)dr
	\le \left\| T_rf \right\|_{L^{4q/(q-2)}(L^{q})}^{q} \|\psi\|_{L^{4/(6-q)}}
	\lesssim
		\|f\|^{q}  \|\psi\|_{L^{4/(6-q)}}. \notag
\eeq
\end{proof}

Similar to Proposition 2.15 in \cite{ChoiHuLee2016}, one can
extend the bound \eqref{eq:L^2 boundedness} to $q>6$ for $f \in H^1(\R)$.
In the following we use $a_+=\max(a,0)$ for the positive part of
$a\in\R$.
\begin{lemma}
\label{lem:H^1-boundedness}
	Let $2\leq q < \infty $ and $\psi \ge 0$ in $L^{\f{4}{6-q+\kappa}}(\R)$ for some
	$(q-6)_+\le \kappa\le q-2$.  Then for all $f\in H^1(\R)$
	\begin{align}\label{eq:H^1-boundedness}
		\|T_r f\|_{L^q(\R^2,dx\psi dr)}^q
        \lesssim \| f'\|^{\f{\kappa}{2}}\|f\|^{q-\f{\kappa}{2}},  \notag
	\end{align}
	where the implicit constant depends only the $L^{\f{4}{6-q+\kappa}}$ norm of $\psi$. 	
\end{lemma}

\begin{proof} This can be found in the proof of Proposition  2.15
in \cite{ChoiHuLee2016}. For the reader's convenience, we give
the short proof:
Since $2\le q-\kappa\le 6$ and $\psi \in L^{\f{4}{6-(q-\kappa)}}(\R)$, applying Lemma
  \ref{lem:L^2boundedness}, we get
  \beq\label{ineq:for H^1}
  \begin{aligned}
  	\iint_{\R^2} |T_rf|^q \, dx \psi(r) dr
  	&\le \sup_{r\in\R }\|T_rf\|_{L^\infty}^\kappa
  		\iint_{\R^2} |T_rf|^{q-\kappa} \, dx \psi(r) dr \\
  & \lesssim \sup_{r\in\R }\| T_rf\|_{L^\infty}^\kappa \|f\|^{q-\kappa}.
  \end{aligned}
  \eeq
Now using the well--known bound
\begin{equation} \label{eq:Kato}
\|g\|_{L^\infty}^2\leq \|g'\|\|g\|,  \notag
\end{equation}
which follows easily from
\[
|g(x)|^2 =2\re \int_{-\infty}^x \ol{g(t)}g'(t)\, \mathrm{d}t = - 2\re \int_x^\infty \ol{g(t)}g'(t)\, \mathrm{d}t
\]
for all $g \in H^1(\R)$ and $x\in\R$, we obtain
\begin{align}\label{ineq:L infinity bound}
  \sup_{r\in\R }\|T_rf\|_{L^\infty}^2
  \leq \sup_{r\in \R}  \|\partial_x(T_r f)\|\|T_r f\|
	=\|f'\|\|f\| ,
\end{align}
where we used the fact that $\partial _x$ and $T_r=e^{ir\partial_x^2}$ commute and $T_r$ is unitary on $L^2(\R)$.
Combining \eqref{ineq:for H^1} and \eqref{ineq:L infinity bound} completes the proof.
\end{proof}

\begin{remark}
It immediately follows from Lemma \ref{lem:H^1-boundedness} that
$$
\|T_r f\|_{L^q(\R^2,dx\psi dr)}\lesssim \|f\|_{H^1}
$$
under the conditions in Lemma \ref{lem:H^1-boundedness}.
A similar argument shows that if $2\leq  q \leq \infty$ and $f \in H^1(\R)$, then
\beq \label{est:H^1bound}
\| T_r f\|_{L^q} \leq \|f\|_{H^1} \notag
\eeq
for arbitrary $r\in \R$. Indeed, this bound clearly holds due to \eqref{ineq:L infinity bound} when $q=\infty$.
For $2\leq q< \infty$, we have
\bdm
\int_{\R} |T_r f| ^q dx\leq  \|T_r f\|_{L^\infty}^{q-2}\int_{\R} |T_r f| ^2 dx \leq \|f\|^{\f{q+2}{2}}\|f'\| ^{\f{q-2}{2}}\leq \|f\|_{H^1}^q.
\edm
\end{remark}

We denote the nonlocal nonlinearity in \eqref{eq:main} by
\beq\label{def:nonlinearity}
Q(f):=\int_\R T_r^{-1}(P(T_r f) )\psi(r)dr
\eeq
for $f$ in either $L^2(\R)$ or $H^1(\R)$.
Then the map $f \mapsto Q(f)$ is bounded and locally Lipschitz continuous
as in the following two lemmas.

\begin{lemma}\label{lem:L^2 bound}
Suppose that $h$ satisfies assumption \ref{ass: dav zero} and
$0\le \psi \in  {L^{1}(\R)}\cap L^{\f{4}{4-p}}(\R)$.
Then for all $f, g \in L^2(\R)$ we have
\beq\label{ineq:L2 boundness for Q_2}
\| Q(f)\| \lesssim  \|f\|+\|f\|^{p+1}
\eeq
and
\beq\label{ineq:L^2 boundness for Q_2}
\| Q(f)-Q(g)\| \lesssim  \left(1+ \|f\|^{p}+\|g\|^{p} \right)\|f-g\|,
\eeq
where the implicit constants depend only on $p$ and the $L^{1}$, $ L^{\f{4}{4-p}}$ norms of $\psi$.
\end{lemma}
\begin{proof}
Since $P(z)= h(|z|)z$, the triangle inequality for integrals implies
\begin{align*}
\|Q(f)\|&
\leq \int_\R \|T_r^{-1}(P(T_rf))\|\psi(r) dr \\
 & \lesssim  \int_\R \left( \|T_rf\|+ \||T_rf|^{p+1}\|\right) \psi(r)dr \, ,
  \end{align*}
  where we used assumption \ref{ass: dav zero}.
  For the first term,
  we note that $\|T_rf\|= \|f\|$, since $T_r$ is unitary on
  $L^2(\R)$. For the second term, we use H\"{o}lder's inequality with
  exponents $\f{4}{p}$ and $\f{4}{4-p}$ in $r$ to get
  \begin{align*}
    \int_\R \||T_rf|^{p+1}\| \psi(r)dr &=\int_\R \|T_rf\|_{L^{2(p+1)}}^{p+1} \psi(r) dr \\
    &\leq \left(\int_\R \|T_rf\|_{L^{2(p+1)}}^{\f{4(p+1)}{p}}dr \right) ^{\f{p}{4}} \left(\int_\R | \psi(r)|^{\f{4}{4-p}} dr\right)^{\f{4-p}{4}} .
  \end{align*}
Thus \eqref{ineq:L2 boundness for Q_2} follows from the Strichartz estimate \eqref{est:strichartz ineq}.

For the second bound, we again use the triangle inequality and the unitarity of $T_r$ on $L^2(\R)$ to see that
\begin{align*}
\|Q(f)-Q(g)\| & \leq \int _{\R}\|P(T_r f)- P(T_r g)\| \psi(r)dr \, .
\end{align*}
Let $w,z\in\C$. From assumption \ref{ass: dav zero} one gets for any $0\le s\le 1$
\begin{align*}
	\Big|\f{d}{ds}P(w+&s(z-w))\Big| = \Big|\f{d}{ds}\big[h(|w+s(z-w)|)(w+s(z-w))\big]\Big| \\
		&\le \big|h'(|w+s(z-w)|)\big| |w+s(z-w)||z-w| + \big|h(|w+s(z-w)|)\big||z-w| \\
		&\lesssim \big(1+|w+s(z-w)|^p\big) |z-w| \le \big(1+\max(|w|,|z|)^p\big)|z-w|
\end{align*}
and the fundamental theorem of calculus gives for all $z,w\in \C$
\begin{equation}\label{ineq:fundamental theorem}
\begin{split}
		|P(z)-&P(w)| = \left| \int_0^1 \f{d}{ds}\Bigl(P(w+s(z-w))\Bigr)ds\right|
 		\lesssim \big(1+\max(|w|,|z|)^p\big) |z-w|\, .
 \end{split}
\end{equation}
Therefore
\begin{equation}
\begin{split}\label{ineq:local Lipschitz in L^2}
	&\|Q(f)-Q(g)\|  \lesssim  \int _{\R} \left\|\Bigl(1+ \max(|T_r f|,|T_r g|)^{p}\Bigr)|T_r (f-g)|\right\| \psi(r)dr \\
& \leq  \int _{\R} \|T_r(f-g)\|\psi(r)dr  + \int _{\R}\left\| \left( |T_r f|^{p}+|T_r g|^{p}\right) T_r(f-g)\right\| \psi(r) dr.
\end{split}	
\end{equation}
Note that the first term equals $\|f-g\|\|\psi\|_{L^1}$. If $p=0$, the second term is bounded in the same way. So to control the second term, it is enough to assume that $0<p\le 4$. 
Use  H\"{o}lder's inequality with $\alpha$ and $\f{2\alpha}{\alpha-2}$
in $x$ to get
\begin{align*}
\| |T_r f|^{p} T_r(f-g)\|
  & \leq
\|| T_r f|^{p}\|_{L^{\alpha}}  \|T_r (f-g)\|_{L^{\f{2\alpha}{\alpha-2}}}
  = \| T_r f\|^{p}_{L^{\alpha p}}  \|T_r (f-g)\|_{L^{\f{2\alpha}{\alpha-2}}}.
\end{align*}
Note that $\frac{2\alpha}{\alpha-2}>2$ for any $\alpha>2$ and one can
always choose $\alpha>2$ such that also $\alpha p \ge 2$. Fix such an
$\alpha >2$ and  use H\"{o}lder's inequality with three exponents
$\f{4\alpha}{\alpha p -2}$, $2\alpha$ and $\f{4}{4-p}$ in $r$
to obtain
\begin{align}\label{ineq:lipschitz Q}
  \int _{\R}\| |T_r f|^{p} &T_r(f-g)\| \psi(r)dr
   \leq \int _{\R} \| T_r f\|^{p}_{L^{\alpha p}}  \|T_r (f-g)\|_{L^{\f{2\alpha}{\alpha-2}}} \psi(r) dr  \notag \\  \notag
 \leq&
\left(\int _{\R} \| T_r f\|^{\f{4\alpha p}{\alpha p-2}}_{L^{\alpha p}} dr \right)^{\f{\alpha p-2}{4\alpha}}
  \left(\int _{\R} \| T_r (f-g)\|^{2\alpha}_{L^{\f{2\alpha}{\alpha-2}}}dr\right )^{\f{1}{2\alpha}}
 \left(\int _{\R} |\psi(r)|^{\f{4}{4-p}}dr\right)^{\f{4-p}{4}} \\ \notag
 \lesssim & \|f\|^{p}\|f-g\| \|\psi\|_{L^{\f{4}{4-p}}},
\end{align}
where we used the Strichartz estimate for the first two factors.
Using this in \eqref{ineq:local Lipschitz in L^2} proves the second part of the lemma.
\end{proof}

\begin{lemma} \label{lem:nonlinear H^1 bound}
Suppose that $h$ satisfies assumption \ref{ass: dav not zero} and $\psi \geq 0$ in $L^1(\R)$. Then for all $f, g \in H^1(\R)$ we have
\beq \label{ineq:boundedness in H^1}
\| Q(f)\|_{H^1} \lesssim \Bigl[J_1 (\|f\|_{H^1})+ J_2 (\|f\|_{H^1})(1+ \|f\|_{H^1})\Bigr]\|f\|_{H^1}  \notag
\eeq
and with $a\vee b = \max(a,b)$ for real numbers $a$ and $b$
\beq\label{ineq:nonlinear L^2_H^1 estimate}
\begin{aligned}
&\| Q(f)-Q(g)\|\\
&\lesssim \Bigl[J_1 (\|f\|_{H^1}\vee \|g\|_{H^1}) +J_2 (\|f\|_{H^1}\vee \|g\|_{H^1}) (1+ \|f\|_{H^1}\vee \|g\|_{H^1})\Bigr]\|f-g\|,
\end{aligned}
\eeq
where the implicit constants depend only on the $L^1$ norm of $\psi$.

\end{lemma}
\begin{proof}
Let $f\in H^1(\R)$. We first show
\beq \label{ineq:boundedness}
\| Q(f)\| \leq J_1 (\|f\|_{H^1})\|f\|\|\psi\|_{L^1}.  \notag
\eeq
Use the triangle inequality, the unitarity of $T_r$ on $L^2(\R)$, and assumption \ref{ass: dav not zero} to get
\beq \label{ineq:Minkowski-H^1}
\begin{aligned}
\|Q(f)\|
&\leq \int_\R \left\|P(T_rf)\right\| |\psi(r)| dr \leq \int_\R \left\|J_1(|T_rf|)\right\|_{L^\infty}\|T_r f\| \psi(r) dr\\
&\leq \int_\R J_1(\|T_rf\|_{L^\infty})\|T_r f\| \psi(r) dr
	 \leq J_1(\|f\|_{H^1}) \|f\| \|\psi\|_{L^1},
\end{aligned}
\eeq
where we also used the assumption that $J_1$ is increasing and $\|T_r f\|_{L^\infty} \leq \|T_r f\|_{H^1}=  \|f\|_{H^1}$.

For any $g\in H^1(\R)$
\begin{align*}
	\big|\partial_x P(g)\big|  &= \big|\partial_x h(|g|)g \big|
	= \big|h(|g|)g'+ h'(|g|)\re \left(\f{\overline{g}}{|g|}g'\right) g \big| \\
		&\leq |h(|g|)g'|+ |h'(|g|)||g'| |g|
		\leq \bigl[J_1(|g|)+ J_2(|g|)(1+ |g|)\bigr]|g'|,
\end{align*}
where we used assumption \ref{ass: dav not zero}.
Since $J_1$ and $J_2$ are increasing, we get
\begin{align*}
	\left\|\partial_x\bigl(P(T_rf)\bigr)\right\|
	&\leq  \| J_1(|T_r f|)+ J_2(|T_r f|)(1+ |T_r f|)\| _{L^\infty} \|\partial _x T_rf\|\\
	& \leq \Bigl[J_1(\| f\| _{H^1}) + J_2(\| f\| _{H^1}) (1+ \| f\| _{H^1})\Bigr] \| f'\|.
\end{align*}
From this we obtain
\begin{align*}
\big\|\partial_x Q(f)\big\|
&\leq \int_\R \left\|\partial_x P(T_rf)\right\| \psi(r)dr\\
&\leq \Bigl[J_1(\| f\| _{H^1}) + J_2(\| f\| _{H^1})(1+ \| f\| _{H^1})\Bigr]  \|f'\|\|\psi\|_{L^1}
\end{align*}
which together with \eqref{ineq:Minkowski-H^1} proves the first bound of the lemma.

Next, we prove the second bound. Arguing similarly as in
the derivation of \eqref{ineq:fundamental theorem}, we have for
$z, w \in \C$
\begin{align*}
\big|P(z)-P(w)\big|&=\big|h(|z|) z -h(|w|) w\big|\\
&\leq |z-w| \int _0 ^1\Bigl[ |h'(|w+s(z-w)|)|  |w+s(z-w)| +|h( |w+s(z-w)|) | \Bigr]ds\\ \notag
& \leq |z-w| \Bigl[J_1 (|z|\vee |w|)+J_2 (|z|\vee |w|)(1+ |z|\vee |w|)\Bigr],
\end{align*}
where we used assumption \ref{ass: dav not zero} for $h$ in the last bound.
This implies
 \begin{align*}
 & \|Q(f)-Q(g)\|
 	\leq \int_\R \|P(T_r f)- P(T_r g)\| \psi(r)dr \\
 \leq  & \int _{\R} \|J_1(|T_r f|\vee |T_r g|)+J_2(|T_r f|\vee |T_r g|) (1+ |T_r f|\vee |T_r g|)\|_{L^\infty} \|T_r(f-g)\|\psi(r) dr.
 \end{align*}
This proves \eqref{ineq:nonlinear L^2_H^1 estimate}, since
$J_1$ and $J_2$ are increasing,
$\|T_r f\|_{L^\infty} \leq \|f\|_{H^1}$, and $T_r$ is unitary.
\end{proof}

\section{Local existence} \label{sec: local existence  for DMNLS}

In this section, we prove the existence of local strong solutions of \eqref{eq:main}, equivalently, local solutions of \eqref{eq: Duhamel u}.
This can be proven with by now  standard arguments (see, for example, \cite{Cazenave,Kato1987}). However, since, in particular in the $H^1$
setting, we want to impose rather weak differentiability conditions on the
nonlinearity, the proofs are somewhat technical and we prefer to give the proofs in detail for the reader's convenience.

Here and below, we use $C$ to denote various constants.
First, we show the existence of local solutions of \eqref{eq: Duhamel u} in the case of vanishing average dispersion.

\begin{proposition}\label{prop:local existence in L^2}
Let $\dav= 0$. Suppose that $h$ satisfies  assumption \ref{ass: dav zero}
and  $\psi\in L^{1} (\R)\cap L^{\f{4}{4-p}} (\R)$.
Then there exists a unique local solution of \eqref{eq: Duhamel u}.
More precisely, for any $K>0$ there exist positive numbers
$M_\pm$, depending also on $p$ and the $L^{1}$, $L^{\f{4}{4-p}} $
norms of $\psi$, such that for any initial condition $u_0\in L^2(\R)$
with $\|u_0\|\le K$ there exists a unique solution
$u\in \calC([-M_-, M_+], L^2)$ of \eqref{eq: Duhamel u}.  Moreover,
\begin{equation} \label{eq:local L^2 bound}
	\|u(t)\|\le 2K \quad \text{for all } t\in [-M_-,M_+]\, .
\end{equation}
\end{proposition}

An immediate consequence is
\begin{corollary} \label{cor:maximal existence in L^2}
Let $\dav= 0$. Suppose that $h$ satisfies  assumption \ref{ass: dav zero}
 and  $\psi\in L^{1} (\R)\cap L^{\f{4}{4-p}} (\R)$.
For any  initial datum $u_0\in L^2(\R)$ there exists maximal
life times $T_\pm\in (0,\infty]$ such that there is a unique solution
$u\in \calC((-T_-, T_+), L^2)$ of \eqref{eq: Duhamel u}.
 	Moreover, the blow--up alternative for solutions holds: If $T_+<\infty$
  	then
 	\begin{equation}
 		\lim_{t\to T_+} \|u(t)\| =\infty  \notag
 	\end{equation}
 	and similarly, if $T_-<\infty$, then
 	\begin{equation}
 		\lim_{t\to -T_-} \|u(t)\| =\infty .  \notag
 	\end{equation}
\end{corollary}
\begin{remark}
	Due to mass conservation given in
	\eqref{eq:mass conservation vanishing average dispersion},
	Corollary \ref{cor:maximal existence in L^2} immediately
	yields a unique global solution when $\dav=0$.
\end{remark}
\begin{proof}[Proof of Proposition \ref{prop:local existence in L^2}]
We will prove the existence of local solutions for positive times only since the case of negative times is done similarly. Fix $u_0 \in L^2(\R)$ and for each $M>0$ define the map $\Phi$ on $\calC([0,M],L^2)$ by
\bdm
\Phi(u)(t)=u_0 +i\int _0 ^t  Q(u(t')) dt'\, ,
\edm
where $Q$ is defined in \eqref{def:nonlinearity}. It is easy to see that
$\Phi(u)\in \calC([0,M],L^2)$.

 For each $R>0$, define the ball
\begin{equation*}
	B_{M,R}=\{u\in \calC([0, M], L^2)  \; : \; \|u\|_{\calC([0, M],L^2)}\leq R\}\, ,
\end{equation*}
equipped with the distance
\begin{equation*}
d(u,v)=\|u-v\|_{\calC([0, M],L^2)}.
\end{equation*}

For
 appropriate values of $R$ and $M$, the map $\Phi$ is a contraction on
 $B_{M,R}$ with respect to the metric $d$.
Indeed,
Lemma \ref{lem:L^2 bound} shows that there exists a constant $C$ depending only on $p$ and the $L^1$, $L^{\f{4}{4-p}}$ norms of $\psi$ such that for all $f, g \in L^2(\R)$,
\bdm 
\| Q(f)\| \le C (\|f\|+\|f\|^{p+1})
\edm
and
\bdm 
\| Q(f)-Q(g)\|\le C  \bigl(1+ \|f\|^{p}+\|g\|^{p}\bigr)\|f-g\|.
\edm
Thus, if $u, v \in \calC([0,M],L^2)$, then
\begin{align*}
\|\Phi(u)(t)\|& \leq \|u_0\|+\int_0 ^t \|Q(u(t'))\|dt' \\
&\leq \|u_0\| +C\int_0^t \|u(t')\|+\|u(t')\|^{p+1}dt'
\end{align*}
and
\begin{align*}
\|\Phi(u)(t)-\Phi(v)(t)\|
& \leq \int_0^t\|Q(u(t'))-Q(v(t'))\|dt'\\
& \leq   C\int_0^t \bigl(1+ \|u(t')\|^{p}+\|v(t')\|^{p}\bigr) \|u(t')-v(t')\| dt'\, .
\end{align*}
Therefore, for all $u, v\in B_{M,R}$,
\beq \label{ineq:into map_zero}
\|\Phi(u)\|_{\calC([0, M], L^2)} \leq \|u_0\| +CM (R+R^{p+1})
\eeq
and
\beq \label{ineq:contraction_zero}
d(\Phi(u),\Phi(v))\leq C M(1+2R^p)d(u,v).
\eeq
Now assume that $\|u_0\|\le K$, set $R=2K$, and choose $M_+>0$ satisfying
\beq \label{choice:M-1}
CM_+(1+(2K)^{p}) <  \f{1}{2}.  \notag
\eeq
Then using \eqref{ineq:into map_zero} and \eqref{ineq:contraction_zero},
we conclude that $\Phi$ is a contraction from $B_{M_+,2K}$ into itself
and since $B_{M_+,2K}$ is complete, Banach's contraction mapping theorem
shows that there exists a unique solution $u$ of
\eqref{eq: Duhamel u} in $B_{M_+,2K}$. This also proves \eqref{eq:local L^2 bound}.
\end{proof}
\begin{remark}
	By standard arguments, the contraction mapping also yields that
	on compact time intervals the
	solution depends continuously on the initial condition.
	A more quantitative bound is derivable with the help of a Gronwall argument, see Proposition \ref{prop:continuous dependence}.
\end{remark}

\begin{proof}[Proof of Corollary \ref{cor:maximal existence in L^2} ]
Given an initial datum $u_0\in L^2(\R)$, let
\begin{equation}
T_+=T_+(u_0)= \sup \{M:\, \exists \text{ unique solution } u\in \calC([0,M],L^2)\text{ with } u(0)=u_0\} \, .  \notag
\end{equation}
Proposition \ref{prop:local existence in L^2} shows $T_+>0$ and that
$u$ is the unique solution of \eqref{eq: Duhamel u} with initial datum $u_0$ for all $t\in [0,T_+)$.
To see the blow--up alternative, assume that  $T_+<\infty$, but
\begin{equation*}
K\coloneqq \liminf_{t\to T_+}\|u(t)\|+1<\infty\, .
\end{equation*}
Then there exists a sequence of times $t_n\to T_+$, as $n\to\infty$, with
$\|u(t_n)\|< K$.

By simply shifting in time, the already proven local existence result from Proposition \ref{prop:local existence in L^2} shows that there is a
time $\Delta T$, depending only on $p$ and the $L^{1}$, $L^{\f{4}{4-p}} $ norms of $\psi$, and $K$,
such that there is a unique solution $\wti{u}\in \calC([t_n, t_n+\Delta T], L^2)$
of \eqref{eq: Duhamel u}. This solution agrees with $u$ on the time
interval $[t_n, T_+) $ and thus concatenating these two unique
solutions one gets, for all $n\in\N$, a unique solution $u$ in
$\calC([0, t_n+\Delta T], L^2)$ for the given initial condition
$u_0$ at time $t=0$. Since $t_n+\Delta T> T_+$ for large enough $n$, this contradicts the maximality of the life time interval $[0,T_+)$.
Thus, if  $0<T_+<\infty$ we must have $\lim_{t\to T_+} \|u(t)\|=\infty$. The case of negative times is done similarly.
\end{proof}

Next, we  present the local existence result in $H^1(\R)$ when the average dispersion does not vanish.
\begin{proposition}\label{prop:local existence in H^1}
Let $\dav\neq 0$. If $h$ satisfies assumption \ref{ass: dav not zero} and  $\psi \in L^1(\R)$,
then there exists a unique local solution of \eqref{eq: Duhamel u}.
More precisely, for any $K>0$ there exist positive numbers
$M_\pm$, depending also on the $L^{1}$ norm of $\psi$ and $J_1, J_2$ from assumption \ref{ass: dav not zero},  such that
for any initial condition $u_0\in H^1(\R)$ with $\|u_0\|_{H^1}\leq K$,
there exists a unique solution $u\in \calC([-M_-, M_+], H^1)$
of \eqref{eq: Duhamel u}.  Moreover,
\begin{equation} \label{eq:local H^1 bound}
	\|u(t)\|_{H^1}\le 2K \quad \text{for all } t\in [-M_-,M_+]\, .
\end{equation}
\end{proposition}
\begin{remark}
	The proof of Proposition \ref{prop:local existence in H^1} follows
	a strategy due to Kato \cite{Kato1987}, see also \cite{Cazenave}. It yields existence and uniqueness, but falls short of proving continuous dependence on the initial datum, i.e., it does not yield well--posedness. This is done in Proposition
	\ref{prop:continuous dependence on H^1}.
\end{remark}

As for the case of vanishing average dispersion, an immediate consequence is
\begin{corollary} \label{cor:maximal existence in H^1}
Let $\dav\neq 0$ and $h$ satisfy assumption \ref{ass: dav not zero} and  $\psi \in L^1(\R)$.
For any  initial datum $u_0\in H^1(\R)$ there exist maximal
life times $T_\pm\in(0,\infty]$ such that there is a unique solution
$u\in \calC((-T_-, T_+), H^1)$ of \eqref{eq: Duhamel u}.
 	 Moreover, the blow--up alternative for solutions holds: If $T_+<\infty$
  	then
 	\begin{equation}
 		\lim_{t\to T_+} \|u(t)\|_{H^1} =\infty  \notag
 	\end{equation}
 	and similarly, if $T_-<\infty$, then
 	\begin{equation}
 		\lim_{t\to -T_-} \|u(t)\|_{H^1} =\infty .  \notag
 	\end{equation}
\end{corollary}
Given Proposition \ref{prop:local existence in H^1}, the proof of Corollary \ref{cor:maximal existence in H^1} is a straightforward copy of the proof of Corollary \ref{cor:maximal existence in L^2}. So it is enough to give the
\begin{proof}[Proof of Proposition \ref{prop:local existence in H^1}]
As before, we consider only the case of positive times. For each $M>0$ and $R>0$, let
\beq\label{def:B}
B_{M,R}=\{u\in L^\infty([0, M], H^1)  \; : \; \|u\|_{L^\infty([0, M],H^1)}\leq R\}  \notag
\eeq
be equipped with the distance
\begin{equation*}
	d(u,v)=\|u-v\|_{ L^\infty([0, M],L^2)}.
\end{equation*}
It is easy to see that boundedness in $H^1$ and
convergence in $L^2$ imply convergence in $H^1$. Thus
$(B_{M,R}, d)$ is a complete metric space, even though the
distance $d$ is measured in $L^2$.
Let $K>0$ and $u_0\in H^1(\R)$ with $\|u_0\|_{H^1}\leq K$ be fixed. Define the map $\Phi$ on $B_{M,R}$ by
\bdm
\Phi(u)(t)=e^{it \dav \partial_x^2 }u_0+i\int _0 ^t e^{i(t-t') \dav \partial_x^2 } Q(u(t'))dt'.
\edm
We can apply the same argument in the proof of Proposition \ref{prop:local existence in L^2}, using Lemma \ref{lem:nonlinear H^1 bound} instead of Lemma \ref{lem:L^2 bound}.
Then we see that, for all $u, v\in B_{M,R}$,
\beq \label{ineq:into map nonzero}
\|\Phi(u)\|_{L^\infty([0, M], H^1)} \leq K +CM \Bigl(J_1(R)+J_2(R)(1+R)\Bigr)R  \notag
\eeq
and
\beq \label{ineq:contraction nonzero}
d(\Phi(u),\Phi(v))\leq C M\Bigl(J_1(R)+J_2(R)(1+R)\Bigr)d(u,v).  \notag
\eeq
Now set $R=2K$ and choose $M_+>0$ satisfying
\beq \label{choice:M-2}
C M_+\Bigl(J_1(2K)+J_2(2K)(1+2K)\Bigr)<  \f{1}{2},  \notag
\eeq
then we obtain that $\Phi$ is a contraction from $B_{M_+,2K}$ into itself, so
$\|u\|_{L^\infty([0, M_+],H^1)}\leq 2K$, which shows \eqref{eq:local H^1 bound}.
Moreover, the second part of Lemma \ref{lem:Strichartz} shows that
$u$ is even in $\calC([0,M_+], H^1)$.
\end{proof}

\section{Continuous dependence on the initial data} \label{sec: local well--posedness}

To complete the proof of  Theorem \ref{thm:local well-posedness in H^1} and the local well--posedness part of Theorem \ref{thm:well-posedness in L^2}, we need to show that the solution depends continuously on the initial datum.
To do this for zero average dispersion, we prove that the map
$u_0 \mapsto u(t)$ is locally Lipschitz continuous on $L^2(\R)$ by a
Gronwall argument.
\begin{proposition} \label{prop:continuous dependence}
Let $\dav=0$, $h$ satisfy assumption \ref{ass: dav zero} and $\psi\in L^{1} (\R)\cap L^{\f{4}{4-p}} (\R) $.
Then, for every $K>0$, there exists a positive constant $C$ depending only
 on $K, p $, and the  $L^1$ and $L^{\f{4}{4-p}}$ norms of $\psi$
 such that for  all initial data $u_0,v_0\in L^2(\R)$ with
 $\|u_0\|,\|v_0\|\le K$  we have
 \begin{equation*}
 \|u-v\|_{\calC([-M_-, M_+],L^2)}\leq e^{C \max(M_-, M_+)}\|u_0-v_0\|,
 \end{equation*}
 where $u$ and $v$ are the corresponding local strong solutions
 of \eqref{eq:main} with initial data $u_0, v_0$ on the time interval
 $[-M_-, M_+]$ of existence, guaranteed by Proposition \ref{prop:local existence in L^2}.
\end{proposition}

\begin{proof}
Without loss of generality, we assume that $t\ge 0$. From
\eqref{eq:local L^2 bound} we know that
$\|u(t)\|, \|v(t)\|\le 2K$ for $0\le t\le M_+$.
Since
$$
u(t)-v(t)=u_0 - v_0 +i\int _0 ^t \left(Q(u(t'))-Q(v(t'))\right)dt',
$$
we can use \eqref{ineq:L^2 boundness for Q_2}  and the triangle inequality
for norms and integrals to  obtain
\begin{align*}
\|u(t)-v(t)\|
 &\leq \|u_0-v_0\|+\int_0 ^t \|Q(u(t'))-Q(v(t'))\| dt'\\
& \leq \|u_0-v_0\|+C_1\int_0 ^t \Bigl(1+ \|u(t')\|^{p}+\|v(t')\|^{p}\Bigr) \| u(t')-v(t')\| dt'\\
&\leq \|u_0 -v_0\| +C_1(1+2^{p+1}K^p)\int _0 ^t \| u(t')-v(t')\|dt'
\end{align*}
for $0\le t\le M_+$.
Therefore, setting $C= C_1(1+2^{p+1}K^p)$, it follows from Gronwall's inequality that if $0\le t\leq M_+$, then
\begin{align*}
\|u(t)-v(t)\| \leq e^{C t} \|u_0-v_0\|
\end{align*}
which completes the proof.
\end{proof}
\begin{remark}
	Using that for zero average dispersion one has mass conservation,
	see the beginning of Section \ref{sec: energy conservation},
	the local solutions are global and the above proof yields
	\begin{align*}
		\|u(t)-v(t)\| \leq e^{C_1(1+ \|u_0\|^p+\|v_0\|^p) |t|} \|u_0-v_0\|
	\end{align*}
	for all $t$, where $C_1$ depends only on $p$ and the $L^1$, $L^{\f{4}{4-p}}$
	norms of $\psi$.
\end{remark}

It remains to show continuous dependence on the initial datum when
$\dav \neq 0$.

\begin{proposition}\label{prop:continuous dependence on H^1}
Let $\dav\not=0$, $h$ satisfy assumption \ref{ass: dav not zero}, and
$\psi \in L^1(\R)$. Then the local solution of the Cauchy problem
\eqref{eq:main} depends continuously on the initial datum.
More precisely, if $ \varphi,\varphi_n\in H^1(\R)$
with  $\varphi_n \to \varphi$ in $H^1(\R)$ as $n\to \infty$, then there exists a common time interval $[-M_-, M_+]$ for which
the strong solutions
$u$,  respectively $u_n$, of the Cauchy problem \eqref{eq:main} with
initial data $\varphi$, respectively $\varphi_n$, exist 
and
$$
u_n\to u \text{\;\; in\;\;}
\calC([-M_-,M_+], H^1) \cap \calC^1((-M_-,M_+), H^{-1})
\quad \text{as } n\to \infty\, .
$$
\end{proposition}

\begin{proof}
Choose a positive $K$ such that $\|\varphi\|_{H^1}, \|\varphi_n\|_{H^1}\le K$
for all $n\in\N$.
It is enough to consider only positive times.
Using Proposition \ref{prop:local existence in H^1} we
then know there exists $M_+>0$ such that on $[0,M_+]$ the solutions $u,u_n$
of \eqref{eq:main} with initial data $\varphi, \varphi_n$ exist for all $n$ and
\beq \label{ineq:sup u_n}
\|u(t)\|_{H^1}, \|u_n(t)\|_{H^1}\leq 2 K\,   
\eeq
for all $0\le t \le M_+$.

It suffices to prove
$$
u_n\to u \quad \text{ in } \calC([0,M_+], H^1)
$$
as $n\to \infty$, since then $Q(u_n)$ converges to  $Q(u)$ in
$\calC([0,M_+], L^2)$ by \eqref{ineq:nonlinear L^2_H^1 estimate} and
$\partial_x^2 u_n$ converges to $\partial_x^2 u$
in $\calC([0,M_+], H^{-1})$. Hence
\begin{equation*}
	\partial_t u_n = i\dav \partial_x^2 u_n + iQ(u_n) \to \partial_t u \text{\;\; in\;\;} \calC([0,M_+], H^{-1})
\quad \text{as } n\to \infty \, .
\end{equation*}
 Furthermore, since $u, u_n\in \calC([0,M_+], H^1)$ for all $n\in \N$,
 it is enough to show
\[
u_n \to u  \quad \text{ in } L^\infty([0,M_+], H^1)\, .
\]
Using
\begin{equation} \label{eq:difference}
u_n(t)-u(t)=e^{it \dav \partial_x^2 }(\varphi_n - \varphi) +i\int _0 ^t e^{i(t-t') \dav \partial_x^2 } \big(Q(u_n(t'))-Q(u(t'))\big)dt'
\end{equation}
and similar arguments as in the proof of Proposition \ref{prop:local existence in H^1}, we obtain
\begin{align*}
  &\|u_n-u\|_{L^{\infty}([0, M_+], L^2)}\\
 \leq  & \|\varphi_n - \varphi \|+ CM_+ \Bigl(J_1(2K)+J_2(2K)(1+2K)\Bigr)\|u_n-u\|_{L^{\infty}([0, M_+], L^2)}\\
 \leq & \|\varphi_n - \varphi \|+ \f{1}{2}\|u_n-u\|_{L^{\infty}([0, M_+], L^2)},
\end{align*}
which yields
\beq\label{ineq:u_n converges to u in LinftyL2}
  \|u_n-u\|_{L^{\infty}([0, M_+],L^2)}\leq 2\|\varphi_n -\varphi\|.
\eeq
It remains to get a similar bound on
$\|\partial_x(u_n-u)\|_{L^{\infty}([0, M_+],L^2)}$. Using \eqref{eq:difference}
we also get
\beq\label{ineq:derivative of nonlinear term 1}
\begin{split}
\|\partial_x(u_n-u)(t)\|& \leq \|\varphi_n'-\varphi'\|
+ \int _0 ^t \| \partial_x\big(Q(u_n(t'))-Q(u(t'))\big)\|dt'\\
&\leq \|\varphi_n'-\varphi'\| + \int _0 ^t \int _\R  \| \partial_x\big(P(T_ru_n(t'))-P(T_r u(t'))\big)\| |\psi(r)|  dr dt'.
\end{split}
\eeq
Note that, for any differentiable complex-valued functions $f$ and $g$ on $\R$,
\begin{align*}
&\f{d}{dx} (h(|f|)f-h(|g|)g)\\
= & h(|f|)f' - h(|g|)g' +\f{1}{2} \big[h'(|f|)|f|f' - h'(|g|)|g|g'\big]
  + \f{1}{2} \big[ h'(|f|)|f|^{-1}f^2 \overline{f}' - h'(|g|)|g|^{-1}g^2 \overline{g}'\big]\\
= & \ h(|f|)(f'-g')+(h(|f|)-h(|g|))g'+\f{1}{2}h'(|f|)|f|(f'-g')+ \f{1}{2}\left(h'(|f|)|f|-h'(|g|)|g|\right)g' \\
&\quad  +\f{1}{2}h'(|f|)|f|^{-1}f^2(\overline{f}'-\overline{g}') + \f{1}{2}\left(h'(|f|)|f|^{-1}f^2- (h'(|g|))|g|^{-1}g^2\right)\overline{g}'.
\end{align*}
Thus
\begin{align*}
\left|\f{d}{dx} \big[h(|f|)f-h(|g|)g\big]\right|
	&\le  \ \Big(\big|h(|f|)\big| + \big|h'(|f|)f\big| \Big)\big|f'-g'|
	+\big|h(|f|)-h(|g|)\big| |g'| \\
	&\phantom{\le ~} +\big|h'(|f|)|f|-h'(|g|)|g|\big||g'| + \big|h'(|f|)|f|^{-1}f^2-h'(|g|)|g|^{-1}g^2\big||g'|.
\end{align*}
We apply this in \eqref{ineq:derivative of nonlinear term 1} to get
\beq\label{est:derivative}
\begin{aligned}
& \| \partial_x (u_n-u)(t)\|\\
 & \le  \ \|\varphi_n'-\varphi'\| \\
 &\phantom{\le~~}+ \int_0 ^t\int_\R
  \left\| \big[|h(|T_r u_n(t')|)|+ |h'(|T_r u_n(t')|)T_ru_n(t')|\big]\partial_x (T_ru_n-T_ru)(t') \right\||\psi(r)|  dr dt' \\
 &\phantom{\le ~~}+  \int _\R  \left\|\big[h(|T_r u_n|)-h(|T_ru|)\big]\partial_x (T_r u)
\right\|_{L^1([0,M_+],L^2)}|\psi(r)| dr\\
 &\phantom{\le ~~}+    \int _\R  \left\|\left[h'(|T_r u_n|)|T_r u_n|-h'(|T_ru|)|T_ru|\right]\partial_x (T_r u)
\right\|_{L^1([0,M_+],L^2)}|\psi(r)| dr\\
 &\phantom{\le ~~}+    \int _\R \left\|\left[h'(|T_r u_n|)|T_r u_n|^{-1}(T_r u_n)^2-h'(|T_ru|)|T_ru|^{-1}(T_r u)^2\right]\partial_x( T_r u) \right\|_{L^1([0,M_+],L^2)}|\psi(r)|  dr.
\end{aligned}
\eeq
Note
\beq\label{eq: bounded in L infinity}
\|T_r u_n\|_{L^{\infty}([0, M_+], L^{\infty})}
	\le \|T_r u_n\|_{L^{\infty}([0, M _+],H^1)} = \|u_n\|_{L^{\infty}([0, M _+],H^1)}\le 2K
\eeq
for all $n$ and $r\in \R$ because of \eqref{ineq:sup u_n}. So
\begin{equation*}
	 \||h(|T_r u_n|)|+ |h'(|T_r u_n|)T_ru_n\|_{L^\infty([0,M_+],L^\infty)}
	 \le   J_1(2K)+ J_2(2K)(1+2K)\, ,
\end{equation*}
hence the first integral in \eqref{est:derivative} is bounded:
\beq\label{ineq:cont_dependence}
\begin{aligned}
	\int_0 ^t\int_\R &
\left\|\big[|h(|T_r u_n(t')|)|+ |h'(|T_r u_n(t')|)T_ru_n(t')|\big]\partial_x (T_ru_n-T_ru)(t')\right\||\psi(r)|  dr dt'
	\\
	& \le  M_+ \Big(J_1(2K)+ J_2(2K)(1+2K)\Big)\|\psi\|_{L^1}\|\partial_x (u_n-u)\|_{L^\infty([0, M_+],L^2)}.  \notag
\end{aligned}
\eeq
For the second integral in \eqref{est:derivative}, use \eqref{eq: bounded in L infinity} to obtain
\begin{equation} \label{eq:majorant}
\left|\big[h(|T_r u_n|)-h(|T_ru|)\big]\partial_x (T_r u)
\right|
\le (J_1(|T_r u_n|)+J_1(|T_r u|))\, \left|\partial_x T_r u  \right|
 \le  2J_1(2K)\, \left| \partial_x T_r u  \right|.
\end{equation}
 Then
 \begin{align*} 	
 	\left\|\big[h(|T_r u_n|)-h(|T_ru|)\big]\partial_x (T_r u)
	\right\|_{L^1([0,M_+],L^2)}
	&\lesssim \left\|\partial_x (T_r u) \right\|_{L^1([0,M_+],L^2)} \\
	&= \left\|\partial_x u \right\|_{L^1([0,M_+],L^2)}
\le 2K M_+\, .
 \end{align*}
Thus, since $\psi\in L^1(\R)$, by the dominated convergence theorem, it is enough to show
 \begin{equation}\label{eq:convergence}
 	\lim_{n\to\infty}
 		\left\|\big[h(|T_r u_n|)-h(|T_ru|)\big]\partial_x (T_r u)
	\right\|_{L^1([0,M_+],L^2)} =0
\end{equation}
for almost every $r\in\R $ to conclude that the third integral in \eqref{est:derivative} converges to zero as $n\to\infty$.\\
Fix $r\in \R$. Then
\bdm
\|T_r u_n -T_r u\|_{L^\infty([0, M_+], L^2)}=\|u_n-u\|_{L^\infty([0, M_+], L^2)}\leq 2\|\varphi_n-\varphi\|,
\edm
where we used \eqref{ineq:u_n converges to u in LinftyL2}.
Therefore, for almost all $(x,t) \in \R \times [0, M_+]$, $T_ru_n \to T_ru$ as $n \to \infty$. Hence $ h(|T_r u_n|)-h(|T_ru|) \to 0$ as $n\to\infty$,
since $h$ is continuous.  Thus, because of \eqref{eq:majorant} we can use the dominated convergence theorem again to see that \eqref{eq:convergence} holds.\\
This shows
\begin{equation*}
	\lim_{n\to\infty}\int _\R  \left\|\left[h(|T_r u_n|)-h(|T_ru|)\right]\partial_x (T_r u)
\right\|_{L^1([0,M_+],L^2)}|\psi(r)| dr=0\, .
\end{equation*}
To show that the last two integrals in \eqref{est:derivative} converge
to zero as $n\to\infty$,  note that the maps $ z\mapsto h'(|z|)z$ and
$z\mapsto h'(|z|)\frac{z^2}{|z|}$, extended by zero to $z=0$, are
continuous on the complex plane, by assumption.
Moreover,
\bdm
\begin{aligned}
\Big| \Big[h'(|T_r u_n|)|T_r u_n|&-h'(|T_ru|)|T_ru| \Big]\partial_x (T_r u)
\Big|\\
\leq &\left[J_2 (|T_r u_n|)(1+|T_r u_n|)+ J_2(|T_r u|) (1+|T_r u|)\right]  \,
\left|\partial_x T_r u \right|\\
 \leq & J_2(2K) (2+4K) \,\left| \partial_x T_r u  \right|
 \end{aligned}
 \edm
and
\bdm
\begin{aligned}
\left| \left[h'(|T_r u_n|)|T_r u_n|^{-1}(T_r u_n)^2-h'(|T_ru|)|T_ru|^{-1}(T_r u)^2\right] \partial_x( T_r u)\right|
\leq & J_2(2K) (2+4K) \,\left| \partial_x T_r u  \right|
\end{aligned}
\edm
for almost all $(x,t) \in \R \times [0, M_+]$. Thus we
can use the same argument as for the third integral in \eqref{est:derivative}
to show that the last two integrals in \eqref{est:derivative} converge to
zero as $n\to\infty$.

Thus we end up with
\begin{align*}
	\| \partial_x (u_n-u)(t)\|_{L^\infty([0,M_+],L^2)}
		\lesssim \|\varphi_n'-\varphi'\|+
			M_+ \| \partial_x (u_n-u)(t)\|_{L^\infty([0,M_+],L^2)}
		+ \oh_n(1),
\end{align*}
where $\oh_n(1)$ denotes terms which go to zero in the limit $n\to\infty$.
Choosing $M_+$ small enough, we conclude
\begin{equation*}
		\| \partial_x (u_n-u)(t)\|_{L^\infty([0,M_+],L^2)}
		\lesssim \|\varphi_n'-\varphi'\| +  \oh_n(1)\, . \qedhere
\end{equation*}
\end{proof}

\section{Mass and energy conservation } \label{sec: energy conservation}
The usual approach to prove global existence from local existence on
$L^2(\R)$  is to show that the mass
\begin{align} \label{mass-again}
	m(u(t))= \|u(t)\|^2
\end{align}
is conserved. This is easy when the average dispersion vanishes
since then
\begin{align*}
	\dot{u}(t)\coloneqq\partial_t u(t) = iQ(u(t)) \in L^2(\R)
\end{align*}
for any strong solution $u$ of \eqref{eq:main}. Thus $\|u(t)\|^2= \la u(t), u(t) \ra$ is differentiable in $t$ and
\begin{align}\label{eq:mass conservation vanishing average dispersion}
	\frac{d}{dt} \|u(t)\|^2 = 2\re \la u, \dot{ u} \ra  = 2\re(i\la u, Q(u) \ra  ) = 0
\end{align}
Thus $\|u(t)\|^2$ is constant, i.e., the mass is conserved.
The last equality in
\eqref{eq:mass conservation vanishing average dispersion} follows from
\begin{equation}\label{eq:real}
  \begin{split}
	\la u, Q(u) \ra
		&= \int_\R \la u, T_r^{-1}\left(P(u)\right) \ra \, \psi(r) dr
			= \int_\R \la T_r u, P(u) \ra\, \psi(r) dr \\
		&= \iint_{\R^2 } h(|T_ru|)|T_r u|^2 \, dx\, \psi(r) dr
  \end{split}
\end{equation}
which shows that $\la u, Q(u) \ra$ is real.

The conservation of mass when $\dav\not=0$ is a little bit trickier: In order to calculate the derivative of the mass one would like to argue that
\begin{align*}
	\frac{d}{dt} \|u(t)\|^2 &= 2\re  \la u, \dot{ u} \ra
	= 2\re\big(i \la u, \dav\partial_x^2u  \ra + i\la u, Q(u)\ra  \big) \\
	&= 2\re\big(-i\dav\la \partial_x u, \partial_x u  \ra + i\la u, Q(u)\ra\big)
	= 0
\end{align*}
since both $\la \partial_x u, \partial_x u  \ra$ and $\la u, Q(u) \ra$ are real.
This argument misses, however, that $\partial_t u\in H^{-1}(\R)$, so
$\la u ,\partial_t u\ra $ is not defined.

While the above  argument
can be saved using that $u\in H^1(\R)$, so the pairing of $u$ and
$\partial_x^2 u$ is well--defined, the problem is much more pronounced, when
one tries to prove conservation of the energy
\begin{align} \label{energy-again}
	E(u(t))= \f{\dav}{2} \|\partial_x u (t)\|^2-\iint_{\R^2} V(|T_r u(t)|)\,dx\psi(r)dr\,
\end{align}
as a first step in order to from local existence to global existence.
Here, $V$ is the antiderivative of the nonlinearity $P$ with $V(0)=0$,
i.e., $V(a)=\int_0^a P(s)\, \mathrm{d}s$ for $a\in \R_+$.

In this case, the derivative of the kinetic energy of $u$ is \emph{not well--defined} since, informally
\begin{align*}
	\frac{d}{dt} \|\partial_x u (t)\|^2
    =2\re \la \partial_x u,\partial_x \partial_t u\ra
	= 2\re  \big(i\dav \la \partial_x u, \partial_x\partial_x^2 u \ra
	+ i\la \partial_x u, \partial_x Q(u)\ra  \big) \, .
\end{align*}
However, since $\partial_x u\in L^2(\R)$ and $\partial_x^3u\in H^{-2}(\R)$, the scalar product $\la \partial_x u, \partial_x\partial_x^2 u \ra$ is not defined anymore.

In order to circumvent this problem, one usually approximates the
solution $u$ by smooth ones and uses an approximation argument.
Following this route, one has
to study solutions of \eqref{eq:main} for initial condition in
Sobolev spaces $H^s(\R)$ with high enough regularity $s>1$.  This
poses additional conditions on the nonlinearity, in particular,
high enough differentiability of $h$, which we need to
avoid. To circumvent this problem we will use the
\emph{twisting trick}  from \cite{AHH, Ozawa}, which
goes back to Dirac's interaction picture in quantum mechanics.

As a warm up, we use the twisting trick to give a simple proof of
mass conservation.
\begin{proposition}[Mass conservation]\label{prop: mass conservation}
  Any solution $u\in\calC([-M_-,M_+], H^1)$ for $\dav\neq 0$, or $u\in\calC([-M_-,M_+], L^2)$ for $\dav= 0$, of the integral equation \eqref{eq: Duhamel u} has conserved mass,
  \begin{equation}\label{mass conservation}
  	\|u(t)\|^2 = \|u_0\|^2 \quad \text{ for all } t\in [-M_-,M_+]\, .
  \end{equation}
\end{proposition}
\begin{proof}
  In order to rigorously show conservation of mass and energy when $\dav\not=0$, we
  twist the solution $u$. In physics this is known as Dyson's interacting picture. Given $u$ let  $v(t)\coloneqq  e^{-it\dav \partial_x^2}u(t)$.
  Then since $u$ solves \eqref{eq: Duhamel u}, $v$ solves
  \begin{equation}\label{eq: twisted Duhamel}
  	 v(t) = u_0 +i \int _0 ^t e^{-it'\dav \partial _x^2} Q(u(t'))\,dt'.
  \end{equation}
  Under the assumptions on the nonlinearity, $Q$ maps $L^2(\R)$ boundedly
  into $L^2(\R)$ for $\dav=0$, respectively, $H^1(\R)$ boundedly into $H^1(\R)$ when
  $\dav\not = 0$. Thus \eqref{eq: twisted Duhamel} shows that $v$ is differentiable with respect to $t$ and
  \begin{equation}\label{eq: derivative v}
  	\dot{v}(t)= \partial_t v(t) =  ie^{-it\dav\partial_x^2}Q(u(t))
  \end{equation}
  is in $L^2(\R)$ when $\dav=0$, respectively, in $H^1(\R)$ when $\dav\not=0$.
  Since $e^{-it\dav\partial_x^2}$ is unitary on $L^2(\R)$, we have
  $\|u(t)\|=\|v(t)\|$ for all $t$, hence
  \begin{align*}
  	\f{d}{dt} \|u(t)\|^2
  		&= \f{d}{dt}\|v(t)\|^2 = 2\re \la v(t), \dot{v}(t) \ra
  			= 2\re\la v(t), ie^{-it\dav\partial_x^2}Q(u(t)) \ra \\
  		&= 2\re(i\la e^{it\dav\partial_x^2}v(t), Q(u(t)) \ra ) = 2\re(i\la u(t), Q(u(t))\ra) =0\,
  \end{align*}
  since \eqref{eq:real} shows that $\la u(t), Q(u(t))\ra$ is real.
Hence the $L^2$ norm of the strong solution $u$ is constant in $t$.
\end{proof}
In the following we abbreviate the nonlocal nonlinearity in \eqref{energy} by
\begin{align}\label{eq:nonlinearity abbreviation}
  N(f) = \iint_{\R^2} V(|T_rf|)\, dx\psi(r) dr 	.
\end{align}
Then the energy of $u$ is given by
\begin{align}
	E(u(t))= \f{\dav}{2} \|\partial_x u (t)\|^2 - N(u(t)) \, .  \notag
\end{align}
\begin{proposition}[Energy conservation, $\dav\not = 0$]\label{prop: energy conservation}
Any solution $u\in\calC([-M_-,M_+], H^1)$ of the
integral equation \eqref{eq: Duhamel u} has conserved energy,
  \begin{equation}\label{energy conservation}
  	E(u(t)) = E(u_0) \quad \text{ for all } t\in [-M_-,M_+]\, .
  \end{equation}
\end{proposition}
\begin{proof}
  We use again the twisted solution $v(t) = e^{-it\dav\partial_x^2}u(t)$.
  Since $e^{-it\dav\partial_x^2}$ commutes with $\partial_x$, we have
 \begin{equation*}
    E(u(t))
  		= \frac{\dav}{2}\|\partial_x u(t)\|^2 - N(u(t))
  		= \frac{\dav}{2}\|\partial_x v(t)\|^2 - N(u(t)).
 \end{equation*}  	
 Using again $\dot{v}(t) = \partial_t v(t) = ie^{-it\dav\partial_x^2}Q(u(t)) $, one sees that the first term is differentiable in $t$ with
  \begin{align} \label{eq: derivative partial v}
  \f{d}{dt} \|\partial_x v(t)\|^2
  		&= 2\re \la \partial_x v(t), \partial_x \dot{v}(t)  \ra
  			=  2\re \la \partial_x  v(t), i\partial_xe^{-it\dav\partial_x^2}Q(u(t))  \ra \nonumber\\
  		&= -2\im \la \partial_x  u(t), \partial_x Q(u(t))  \ra \, .
  \end{align}
  To compute the derivative of the second term, let $w\in \calC^1([-M_-,M_+], H^1)$ and consider $N(w(t))$. The chain rule yields
  \begin{align}
  	\partial_t N(w(t))
  		& = DN(w(t))[\dot{w}(t)]
  			= \iint_{\R^2} V'(|T_r w(t)|) \re \Big(\frac{T_r w(t)}{|T_r w(t)|}\ol{T_r \dot{w}(t)}\Big)  \, dx \psi(r)dr \nonumber\\
  		& = \iint_{\R^2} P(|T_r w(t)|) \re \Big(\frac{T_r w(t)}{|T_r w(t)|}\ol{T_r \dot{w}(t)}\Big)\, dx \psi(r)dr = \re \la \dot{w}(t), Q(w(t)) \ra \nonumber\\
  		&= \re \la (1-\partial_x)^{-1}\dot{w}(t), (1+\partial_x) Q(w(t)) \ra,
  			\label{eq: time derivative N}
  \end{align}
  where we used that $\partial_x$ is skew adjoint, so
  $1-\partial_x:H^1(\R)\to L^2(\R)$ is invertible.

 The right hand side of \eqref{eq: time derivative N} extends to
 $w\in \calC([-M_-,M_+], H^1)\cap\calC^1([-M_-,M_+], H^{-1}) $, by the usual
 density arguments: If $\dot{w}(t)\in H^{-1}(\R)$, then
 $(1-\partial_x)^{-1}\dot{w}(t)\in L^2(\R)$ and $Q(w(t))\in H^1(\R)$,
 so $(1+\partial_x)Q(w(t))\in L^2(\R)$.
 This shows that $N(w(t))$ is
 differentiable in $t$ with derivative given by
 the last line of \eqref{eq: time derivative N} for any
 $w\in \calC([-M_-,M_+], H^1)\cap\calC^1([-M_-,M_+], H^{-1})$.
 \smallskip

 Any solution $u\in\calC([-M_-,M_+], H^1)$ of the
  integral equation \eqref{eq: Duhamel u} has derivative
  \begin{equation}\label{eq: time derivative u}
  	\dot{u}(t)= \partial_t u(t) =i \dav \partial_x^2 u(t)+i Q(u(t))\in H^{-1}(\R) \notag
  \end{equation}
  and the right hand side above is continuous in $t$ with values in
  $H^{-1}(\R)$. So \eqref{eq: time derivative N} applies to $u$ and shows that
  for any solution $u\in \calC([-M_-,M_+], H^1)$ of \eqref{eq: Duhamel u}, $N(u(t))$ is differentiable in $t$ with derivative
  \begin{equation}
  	\begin{split}
  	\partial_t N(u(t))
  		&= \re \la (1-\partial_x)^{-1}\dot{u}(t), (1+\partial_x)Q(u(t)) \ra \\ \notag
  		&= \re \la i(1-\partial_x)^{-1}(\dav\partial_x^2u(t)+ Q(u(t)), (1+\partial_x)Q(u(t)) \ra \\ \notag
  		&= \im   \la (1-\partial_x)^{-1}(\dav\partial_x^2u(t)+ Q(u(t)), (1+\partial_x)Q(u(t)) \ra \, . \notag
  	\end{split}
  \end{equation}
 Note that
 \begin{align*}
 	\la (1-\partial_x)^{-1} Q(u(t)), (1+\partial_x)Q(u(t)) \ra
 		= \la Q(u(t)), Q(u(t)) \ra \in \R
 \end{align*}
 and, since $-\partial_x(1+\partial_x)^{-1}$ is the adjoint of
 $(1-\partial_x)^{-1}\partial_x$ and bounded on $L^2(\R)$,
 \begin{align*}
 	\la (1-\partial_x)^{-1}\partial_x^2u(t), (1+\partial_x)Q(u(t))\ra
 	& = \la \partial_xu(t), -\partial_x(1+\partial_x)^{-1}(1+\partial_x)Q(u(t))\ra \\
 	&  =-\la \partial_x u(t), \partial_xQ(u(t)) \ra\, .
 \end{align*}
 Thus
  \begin{equation}
  	\partial_t N(u(t))
  		= -\dav \im  \la \partial_x u(t), \partial_xQ(u(t))\ra, \notag
  \end{equation}
 which together with \eqref{eq: derivative partial v} proves that the energy $E(u(t))$ is differentiable and
  \begin{align*}
  	\f{d}{dt} E(u(t))
  =0\,
  \end{align*}
  for all $t\in [-M_-,M_+]$.
  Hence the energy is constant.

\end{proof}

\begin{remark}
	In the case of zero average dispersion, which is a kind
	of singular limit, the energy is given by
	$E(f)= - N(f)$ with $N$ defined in
	\eqref{eq:nonlinearity abbreviation}. In this case,
	any strong solution of the dispersion management equation \eqref{eq:main}, or of \eqref{eq: Duhamel u}, satisfies
	energy conservation. This is easy to see: In this case
	$\dot{u}(t)=\partial_tu(t)= i\calQ(u(t))\in L^2$ and a
	calculation leading to
	\eqref{eq: time derivative N}, but now simpler, shows
	\begin{align*}
		\partial_t N(u(t))
			= \re\la \dot{u}(t), \calQ(u(t))\ra
			= \im \la \calQ(u(t)), \calQ(u(t))\ra
			=0 \, .
	\end{align*}

\end{remark}

\section{Global existence } \label{sec: global existence}

In this section, we finish the proof of global well--posedness of the
dispersion managed
NLS \eqref{eq:main}.
We only have to show that solutions exist globally, since the
local existence and uniqueness results tother with the
continuous dependence on the data proved in Section
\ref{sec: local well--posedness} apply to all times for which
the solution exists.
For $\dav=0$, assume that $h$ satisfies assumption \ref{ass: dav zero} and $\psi\in L^{1} (\R)\cap L^{\f{4}{4-p}} (\R) $, then
it follows from the mass conservation  that local solutions
for \eqref{eq: Duhamel u} obtained in Proposition
\ref{prop:local existence in L^2} are bounded in $L^2(\R)$.
Hence the blow--up alternative from Corollary
\ref{cor:maximal existence in L^2} shows that the solution is
global in $t$. This finishes the proof of Theorem
\ref{thm:well-posedness in L^2}.

\begin{proposition}\label{prop:global existence}
 Let $\dav\not=0$,  $h$ satisfy assumptions
 \ref{ass: dav not zero}, \ref{ass: dav not zero global},
and $\psi\in L^1(\R)\cap L^{\frac{4}{4-p}}(\R)$.
 Then  the Cauchy problem \eqref{eq:main} has a unique global strong solution $u$ in
  $\calC(\R, H^1)\cap\calC^1(\R, H^{-1})$  for any initial datum $u_0\in H^1(\R)$.
\end{proposition}
\begin{remark}
	In particular, for negative average dispersion we have a global existence result
	for nonlinearities for which $h$ satisfies assumption \ref{ass: dav not zero} and is bounded from below.
	In applications, the polarization $P(a)=h(a)a,\ a \ge 0,$ is usually
	non--negative, so the requirement that $h$ is bounded from below
	is a rather weak additional condition on the nonlinearity.
	
	If $h$ fulfills a growth condition of the form
	\begin{align}
		|h(a)|\lesssim 1+ a^\beta \quad \text{for } a>0 \notag
	\end{align}
	then it is easy to see that the condition \ref{ass: dav not zero global}
	is fulfilled when
	$0\le \beta< 8$. Other growth conditions such as
	\begin{align}
		|h(a)|\lesssim 1+ a^8(\ln(2+a))^{-1} \quad \text{for } a>0\, \notag
	\end{align}
 	also yield global existence.
\end{remark}

\begin{proof} Since the mass is conserved by Proposition \ref{prop: mass conservation} it is enough to bound $\|\partial_x u(t)\|$ in order to control the $H^1$ norm of the solution $u$.
In order to be able to use energy conservation in Proposition \ref{prop: energy conservation}, we need to control the nonlinearity in a first step.

Recall that the nonlocal nonlinearity is given by
$$
N(f)=\iint_{\R^2} V(|T_r f|)dx \psi(r) dr,
$$
where $V(a)= \int_0^aP(s)\, \mathrm{d}s= \int_0^a h(s)s\, \mathrm{d}s$ for all $a>0$.
Assume $h(a)\le \wti{J}(a)(1+a^p)$, then
\begin{align*}
	V(a) \lesssim   \wti{J}(a)(a^2+a^{p+2}) \text{ for all } a\ge 0,
\end{align*}
and
\begin{align*}
	N(f) \lesssim
			\wti{J}\big((\|f'\|\|f\|)^{1/2}\big)
			\iint_{\R^2} (|T_rf|^2 + |T_rf|^{p+2})\, dx \psi(r)dr
\end{align*}
since $\wti{J}$ is increasing and $|T_rf|\le (\|\partial_x T_rf\|\|T_rf\|)^{1/2}= (\|\partial_x f\|\|f\|)^{1/2}$.
Also, since $T_r$ is unitary on $L^2(\R)$ we have  $\iint_{\R^2} |T_rf|^2 \, dx \psi(r)dr = \|f\|^2 \|\psi\|_{L^1}$.
For the last term we use  Lemma \ref{lem:L^2boundedness} to obtain
\begin{align*}
	\iint_{\R^2} |T_rf|^{p+2}\, dx \psi(r)dr
	\lesssim
		\|f\|^{p+2}  \|\psi\|_{L^{4/(4-p)}} \, .
\end{align*}
Thus
\begin{equation}\label{eq: N upper bound}
		N(f) \lesssim \wti{J}\big((\|f'\|\|f\|)^{1/2}\big) \big(  \|f\|^2 + \|f\|^{p+2}\big)\, .
\end{equation}
In case that $h(a)\ge  -\wti{J}(a)(1+a^p)$, we get similarly
\begin{equation}\label{eq: N lower bound}
		N(f) \gtrsim -\wti{J}\big((\|f'\|\|f\|)^{1/2}\big) \big(  \|f\|^2 + \|f\|^{p+2}\big)\, .
\end{equation}

Let $\dav \neq 0$, then Corollary \ref{cor:maximal existence in H^1}
tells us that there exist $T_\pm>0$ depending only on $\|u_0\|_{H^1}$ and
the $L^1$ norm of $\psi$ such that a  unique solution $u$ for
\eqref{eq: Duhamel u} exists in $\calC((-T_-,T_+),H^1)$ with initial data $u_0\in H^1(\R)$.
Moreover, if $T_+<\infty$, the $H^1$ norm of the solution must blow up as $t\to T_+$ and similarly for $T_-$.

The energy conservation \eqref{energy conservation} shows
\begin{equation}\label{eq: energy conservation made work}
  \begin{split}
 	\|\partial _x u(t)\|^2
 		&=  \f{2E(u_0)}{\dav} + \f{2N(u(t))}{\dav} \\
 		&\le \f{2E(u_0)}{\dav} + C\frac{\wti{J}\big((\|\partial_x u(t)\|\|u_0\|)^{1/2}\big) \big(  \|u_0\|^2 + \|u_0\|^{p+2}\big)}{|\dav|}
  \end{split}
\end{equation}
for some finite positive constant $C$, due to \eqref{eq: N upper bound} when $\dav>0$, respectively \eqref{eq: N lower bound} when $\dav<0$, and
$\|u(t)\|= \|u_0\|$ by conservation of mass \eqref{mass conservation}.

The bound  \eqref{eq: energy conservation made work} is clearly equivalent to
\begin{equation*}
	\|\partial _x u(t)\|^2 \left( 1- C \frac{\wti{J}\big((\|\partial_x u(t)\|\|u_0\|)^{1/2}\big)}{\|\partial _x u(t)\|^2} \right)
	\lesssim 1
\end{equation*}
for all $t\in (-T_-, T_+)$ for some maybe different constant $C$.
Due to the growth condition \eqref{eq: growth condition wtiJ}
on $\wti{J}$ this shows that
$\|\partial _x u(t)\|$ cannot blow up as $t\to T_+$ or $t\to -T_-$.
By mass conservation, this shows that the $H^1$ norm of the solution
does not blow up.
Hence the blow--up alternative from Corollary \ref{cor:maximal existence in H^1} implies that the solution exists globally.
\end{proof}
\begin{remark}
	The above proof shows that under assumption \ref{ass: dav not zero global} we have for $\dav>0$
	\begin{equation}\label{eq:coercivity-1}
		E(f) \ge \frac{\dav}{2}\|f'\|^2 - C \wti{J}\big((\|f'\|\|f\|)^{1/2}\big) \big(  \|f\|^2 + \|f\|^{p+2}\big)\, . \notag
	\end{equation}
 Thus the energy is coercive: for any sequence $f_n\in H^1(\R)$ with
 $\|f_n\|$ bounded and $\|f_n'\|\to \infty$ as $n\to\infty$ one has
 \begin{equation}\label{eq:coercivity-2}
 	\lim_{n\to\infty} E(f_n)= \infty\, .
 \end{equation}
 This will be useful for the proof of orbital stability in
 Section \ref{sec:Stability result}.
\end{remark}

Now we come to the proof of Theorem \ref{thm: small data global well-posedness}, which for the convenience of the reader, we recall.

\begin{proposition}[$=$ Theorem \ref{thm: small data global well-posedness}]\label{prop: small data global existence}
	Let $\dav\not=0$ and $h$ satisfy assumption
	\ref{ass: dav not zero}.
	\begin{theoremlist}
		\item For any initial datum
			$u_0\in H^1(\R)$ with small enough $H^1$ norm,
			 the Cauchy problem \eqref{eq:main} is globally well--posed
			 when $\psi\in L^1(\R)$.
		\item If $J_1(a) \lesssim 1+a^8 $ for $a\ge 0$, then
			 the Cauchy problem \eqref{eq:main} is globally well--posed
			 for initial conditions
  			$u_0\in H^1(\R)$ with $\|u_0\|$ small enough when
  			$\psi\in L^\infty(\R)\cap L^1(\R)$.
		\item If $\lim_{a\to 0}J_1(a)/a^4=0$ then
			 the Cauchy problem \eqref{eq:main} is globally well--posed
			 for initial conditions
  			$u_0\in H^1(\R)$ with $\|u_0'\|$ small enough
  			{\rm(}depending on $\|u_0\|${\rm)} when  $\psi\in L^1(\R)$.
	\end{theoremlist}
\end{proposition}

\begin{proof}
  Since we proved local well--posedness in Sections
  \ref{sec: local existence  for DMNLS} and
  \ref{sec: local well--posedness}, we only have to prove
  global existence of solutions.

  Since $V(a)= \int_0^a h(s)s\, \mathrm{d}s\le J_1(a)a^2/2$
  we can argue similarly to the derivation of \eqref{eq: N upper bound}
  to see that
  \begin{equation}\label{eq:rough N bound}
  	|N(f)| \le \frac{1}{2} J_1\big((\|f'\|\|f\|)^{1/2}\big)
			\iint _{\R^2} |T_rf|^2 \, dx \psi(r)dr
			= \frac{1}{2} J_1\big((\|f'\|\|f\|)^{1/2}\big) \|f\|^2  \|\psi\|_{L^1}\, .
  \end{equation}
Thus energy and mass conservation again yields
\begin{align*}
	\left|\frac{2 E(u_0)}{\dav}\right| \ge \frac{2 E(u_0)}{\dav}
		&= \|\partial_x u(t)\|^2 - \frac{2 N(u(t))}{\dav}
	  		\geq \|\partial_x u(t)\|^2 - \frac{2 |N(u(t))|}{|\dav|} \\
		&\geq \|\partial_x u(t)\|^2 - |\dav|^{-1} J_1\big((\|\partial_x u(t)\|\|u_0\|)^{1/2}\big) \|u_0\|^2
		    \|\psi\|_{L^1}\, .
\end{align*}
Given $\alpha, s\ge 0$, let $G_\alpha(s)= s^2- |\dav|^{-1} \|\psi\|_{L^1}  J_1\big((\alpha s)^{1/2}\big) \alpha^2$. Then the above shows
\begin{equation}\label{eq:a-priori}
	G_{\|u_0\|}(\|\partial_x u(t)\|) \le \left|\frac{2 E(u_0)}{\dav}\right|
\end{equation}
for all $t$ for which the solution $u$ exists.

We will show shortly that the bound \eqref{eq:a-priori} forces
$\|\partial_x u(t)\|$ to stay bounded when the $H^1$ norm of the
initial condition $u_0$ is small enough. Together with the blow--up
alternative from Corollary \ref{cor:maximal existence in H^1} this
shows that the solution is global.

\smallskip
Note that $G_\alpha(s)$ is decreasing in $\alpha\ge 0$ for fixed $s\ge 0$.
In addition,  when $\alpha_0>0$ is small enough, we have
\begin{align}\label{eq:c-alpha0}
	 c_{\alpha_0}=\inf_{0\le \alpha\le \alpha_0}\sup_{s\ge 0} G_\alpha(s)>0 \, .
\end{align}
If $0<c<c_{\alpha_0}$ then there
exist $0<a<b$ such that
$G_{\alpha}(s) > c$ for all $a< s< b$ and all
$0\le \alpha\le \alpha_0$.
Thus $\alpha_0<0$ small enough, $0\le \alpha\le \alpha_0$,
and $G_\alpha(s)\le c$ implies
$0\le s\le a$ or $s\ge b$.
Hence if the initial condition $u_0$ is such that
$\|u_0\|\le \alpha_0$ and $2 E(u_0)/\dav\le c$, then
\eqref{eq:a-priori} implies
$ \|\partial_x u(t)\|\le a$ or $\|\partial_x u(t)\|\ge b$ for all
times for which the solution exists.

Due to \eqref{eq:rough N bound}, we can make $|2E(u_0)/\dav|\le c$
by choosing $\|u_0\|$ and $\|u_0'\|$ small enough. Making them
even smaller, if necessary, we can also assume that
$\|u_0\|\le \alpha_0$ and $\|u'(0)\|\le a$. Since either
$ \|\partial_x u(t)\|\le a$ or $\|\partial_x u(t)\|\ge b$ and
$t\mapsto \|\partial_x u(t)\|$ is continuous,
this shows that $ \|\partial_x u(t)\|\le a$ for all times for which
the solution exists. This finishes the proof of the first part of
the proposition.

For the second part, we note that using Lemma \ref{lem:H^1-boundedness} with $\kappa=4$ and $q=10$ we have
\begin{equation*}
  \begin{split}
	|N(f)|&\lesssim \iint _{\R^2}\left( \|T_rf\|^2 + \|T_r f\|^{10} \right)\, dx \psi(r)dr \lesssim \|f\|^2 + \|f'\|^2 \|f\|^8\, .
  \end{split}
\end{equation*}
Together with energy and mass conservation this implies
\begin{equation}
		 	\|\partial _x u(t)\|^2\left(1-C \|u_0\|^8\right)
 		\le \left|\f{2E(u_0)}{\dav}\right| + C\|u_0\|^2 \,  \notag
 \end{equation}
 similarly as for \eqref{eq:a-priori}.
 Hence, as soon as $\|u_0\|$ is small enough,  the kinetic energy
 $\|\partial _x u(t)\|$ stays bounded,
 so the  blow--up alternative from Corollary
 \ref{cor:maximal existence in H^1} applies again.

 For the proof of the third part, we note that
 the assumption
 $\lim_{a\to 0}J_1(a)/a^4=0$ implies that the constant
 $c_{\alpha_0}$ given by \eqref{eq:c-alpha0} is
 now positive for all $\alpha_0>0$. Moreover,
 for fixed $L^2$ norm of the initial condition $u_0$
 the bound \eqref{eq:rough N bound} together with
 $\lim_{a\to 0}J_1(a)=0$  shows that
 we can make its energy $|E(u_0)|$ as
 small as we like by having $\|u_0'\|$ small.

 Thus, with some small and  straightforward  modifications, the
 proof of the first case now
 shows that $\|\partial_xu(t)\|$ stays bounded for all times for
 which the solution $u$ exists, as long as $\|u_0'\|$ is small enough,
 depending only on how large $\|u_0\|$ is.
 Together with energy and mass conservation the blow--up alternative
 from Corollary \ref{cor:maximal existence in H^1} again shows that
 the solution $u$ is global.
\end{proof}

Once one knows global esistence, a natural next step is a more detailed
investigation of the long time behavior	of the solutions, such as
scattering or the stability of solitary solutions under small
perturbations. Scattering results are easier to
prove in higher
dimensions, since the wave has more directions at its disposal in
order to move to infinity. It can disperse more easily in higher
dimensions than in one dimension, where it can only move to the left or
right and the dispersive effects are the weakest.
Modified scattering
for one dimensional dispersion managed NLS {\blue with a cubic nonlinearity}
was shown for small,
well--localized initial data in \cite{MH1}, see the review paper,
\cite{Murphy-review}, for NLS in one dimension.
In the next section we show orbital stability for the focussing
case, which in our notation means $\dav > 0$, and also vanishing
average dispersion $\dav=0$,  including saturating nonlinearities
for both cases.

\section{Orbital stability for (non--)saturated nonlinearities} \label{sec:Stability result}
In this section, we prove Theorem \ref{thm:Stability result}.
We consider only non--negative average dispersion, $\dav\ge 0$, since the set of
ground states is ill-defined when $\dav<0$.
Recall the set of ground states
$$
S_\lambda^{\dav}=\{f\in X: E(f)=E_{\lambda}^{\dav}, \|f\|^2=\lambda \}
$$
for each $\lambda>0$ and $\dav\geq 0$ and
$$
E_\lambda ^\dav =\inf \{E(f)= \f{\dav}{2} \|f'\|^2 -N(f) : f\in X,  \|f\|^2=\lambda \}.
$$
Here, $X= H^1(\R)$ for $\dav>0$ and $X=L^2(\R)$ for $\dav=0$.\\
Recall that the nonlocal nonlinearity functional is given by
$$
N(f)=\iint_{\R^2} V(|T_rf|) dx \psi(r)dr,
$$
where $V(a) = \int_0^a P(s)\, \mathrm{d}s$, for $a\geq 0$, the antiderivative of $P$,
and the nonlinearity  $ P$ is given by $P(z)= h(|z|)z$ for $z\in \C$.

\smallskip
Recall also the additional assumptions for the orbital stability of $S_\lambda^{\dav}$:

\vspace{2pt}
\begin{assumptions}
\setcounter{assumptions}{\value{saveassumptions}}
\item  There exists $p_0>0$ with
	\begin{equation}\label{eq:strong A-R-again}
		h(a)a^2 \ge p_0 \int_0^a h(s)s\, \mathrm{d}s \text{  for all } a>0\, ,
	\end{equation}
\item
	There exists a continuous decreasing function
	$p:[0,\infty)\to (2,\infty)$ such that
		\begin{equation}\label{eq:saturating A-R-again}			
			h(a)a^2 \ge p(a) \int_0^a h(s)s\, \mathrm{d}s\text{ for all } a>0.
		\end{equation}
\item There exists $a_0>0$ with
	$\int_0^{a_0} h(s)s\, \mathrm{d}s>0$.
\end{assumptions}

Our first result concerns the question whether $S^\dav_\lambda$ is empty or not.
\begin{theorem}[Existence of thresholds for $S^{\dav}_\lambda\neq \emptyset$] \label{thm:existence}
   Suppose that the nonlinearity $h$ satisfies assumption
   \ref{ass:A6}   and
  either of the following:
 \begin{theoremlist}
	\item \textbf{Zero average dispersion, non--saturated nonlinearity:}
		The nonlinearity $h$ satisfies assumption
		\ref{ass:A4} and the bound
		$|h(a)|\lesssim a^{p_1}+ a^{p_2}$
		for some $0< p_1\le  p_2<4$.
		The density
		$\psi\in L^{\frac{4}{4-p_2}+}(\R)$ has compact support.
	\item \textbf{Zero average dispersion, saturated nonlinearity:}
		The nonlinearity $h$ satisfies assumption
		\ref{ass:A5} and the bound
		$|h(a)|\lesssim a^{p_1}+ a^{p_2}$
		for some $1\le  p_1\le  p_2< 3$.
		The density
		$\psi\in L^{\frac{4}{3-p_2}+}(\R)$ has compact support.
	\item    \textbf{Positive average dispersion, saturated and non--saturated nonlinearities:}
		The nonlinearity $h$ satisfies the bound
		$|h(a)|\lesssim a^{p_1}+ a^{p_2}$ for some
		$0<p_1\le p_2<8$.
		The density  $\psi\in L^{\frac{4}{8-p_2}+}(\R)$
	    has compact support.
	    Moreover, $h$ satisfies either assumption
		\ref{ass:A4} or assumption \ref{ass:A5}.
  \end{theoremlist}
  Then there exists a critical threshold $0\le \lambda_{\text{cr}}<\infty$ such that
  if $\lambda >  \lambda_{\text{cr}}^{\dav}$ then $E^\dav_\lambda<0$ and the set $S_\lambda^{\dav}$ is
  not empty. Moreover, if there exists $\veps>0$ such that $h(a)>0$ for $0<a\le \veps$ when $\dav=0$ or that  $h(a)\gtrsim a^{q} $ for $0<a\le \veps$ and some $0<q<4$ when $\dav>0$, then  $\lambda_{\text{cr}}^{\dav}=0$.

  If $\lambda_{\text{cr}}^{\dav}>0$ and $\dav>0$, then $S_\lambda^{\dav}=\emptyset$ for all $0<\lambda<\lambda_{\text{cr}}^{\dav}$.
\end{theorem}
\begin{proof}
  These results can be found for non--saturating nonlinearities in
  \cite{ChoiHuLee2016} and for saturating nonlinearities
  in \cite{HLRZ}. If the average dispersion is zero and the
  nonlinearity is saturating, the condition $1\le p< 3$ is
  needed to guarantee that $S^0_\lambda$  is not empty, at least
  for large enough $\lambda$.
\end{proof}
\begin{remarks}\label{rem:nonempty set of ground states}
  \begin{theoremlist}
    \item The requirement that $\psi$ has compact support is very natural from the point of view of applications for dispersion managed NLS, see Section \ref{sec:connection}.
    \item \label{rem:existence-ground-states}
 The condition on $V(a)=\int_0^a h(s)s\, \mathrm{d}s$ used in \cite{ChoiHuLee2016, HLRZ}
 is $|V'(a)|\lesssim a^{\gamma_1-1} + a^{\gamma_2-1}$ for some
 $2\le \gamma_1\le \gamma_2<10$ for $\dav>0$, and $2< \gamma_1\le \gamma_2<6$
 for $\dav=0$ in \cite{ChoiHuLee2016}.
 The conditions in Theorem \ref{thm:existence} are a bit more general than the assumptions used in \cite{ChoiHuLee2016,HLRZ}.
 However, the proofs carry over to our more general situation:
 The main tools for the proofs in \cite{ChoiHuLee2016,HLRZ}
 that $S^\dav_\lambda$ is not empty are tightness results,
 modulo translations, for
 energy minimizing sequences, see
 \cite[Proposition 4.4 and 4.6]{ChoiHuLee2016}.
 These bounds follow from the strict
 subadditivity of the energy and the splitting bounds for the nonlocal
 nonlinearity in \cite[Section 2.2]{ChoiHuLee2016}.
 This strict subadditivity is shown in \cite{ChoiHuLee2016} under assumption \ref{ass:A4} and in \cite{HLRZ} under assumption \ref{ass:A5}.
 The splitting bounds relied on the pointwise bounds for $V$ from
 \cite[Lemma 2.14]{ChoiHuLee2016}.
 Under the conditions of Theorem \ref{thm:existence} suitable
 replacements, which are sufficient for us, still holds.
 For example, we have
 \begin{equation}\label{eq:splitting V}
 	\left| V(|z+w|) - V(|z|) - V(|w|) \right|
 	\le 4 J_1\big(|z|+|w|\big) | z | |w|
 \end{equation}
 for all $z,w\in\C$, which is a suitable replacement for the bound in equation (2.17) from \cite[Lemma 2.14]{ChoiHuLee2016}. To prove \eqref{eq:splitting V} just argue as in the
 proof of  \cite[Lemma 2.14]{ChoiHuLee2016} using now
\beq
 	\left|V(|z+w|) -V(|z|) \right| =\Big| \int_{|z|}^{|z+w|} h(s)s\, \mathrm{d}s  \Big|
 		\le J_1\big(|z|+|w|\big)\big(|z|+|w|\big) |w| \, . \notag
 \eeq
  \end{theoremlist}
\end{remarks}
To prove Theorem \ref{thm:Stability result}, we need one more result
which is the continuity of the nonlinear functional $N$ similar to \cite[Lemma 4.7]{ChoiHuLee2016}.

\begin{lemma} \label{lem:continuity}
\itemthm  Assume that $h$ satisfies $|h(a)|\lesssim 1+a^p$ for all $a\ge 0$
	and some $0\le p\le 4$ and $\psi \geq 0$ in $ L^1(\R)\cap L^{\frac{4}{4-p}}(\R)$. Then
	the nonlinear nonlocal functional $N: L^2(\R)\to \R$ given by
\begin{align*}
	L^2(\R)\ni f\mapsto N(f)=
	\iint_{\R^2} V(|T_r f|)dx\psi(r) dr
\end{align*}
 is locally Lipshitz continuous on $L^2(\R)$ in the sense that
 \begin{align}\label{ineq:lipschitzL^2}
   	|N(f_1)-N(f_2)| \lesssim \left(\|f_1\|+\|f_2\|+ \|f_1\|^{p+1} + \|f_2\|^{p+1}\right)
   	\|f_1-f_2\|,
 \end{align}
where the implicit constant depends only on $p$ and the $L^1$, $L^{\f{4}{4-p}}$ norms of $\psi$. \\
 \itemthm Assume that $h$ satisfies $|h(a)|\le J_1(a)$ for all $a\ge 0$ and some increasing function $J_1\ge0$ and $\psi\geq0$ in $L^1(\R)$. Then the nonlinear nonlocal functional
 	$N: H^1(\R)\to \R$ given by
\begin{align*}
H^1(\R)\ni f\mapsto N(f)=
\iint_{\R^2} V(|T_r f|)dx\psi(r) dr
\end{align*}is locally Lipschitz continuous in the sense that
 \begin{align} \label{ineq:lipschitzH^1}
 |N(f_1)-N(f_2)| \lesssim J_1\big(\|f_1\|_{H^1}\vee\|f_2\|_{H^1}\big)(\|f_1\|+\|f_2\|)
 \|f_1-f_2\|,
 \end{align}
 where the implicit constant depends only on the $L^{1}$ norm of $\psi$. \\
\end{lemma}
\begin{remark}
	Note that the second part of Lemma \ref{lem:continuity} shows a
	somewhat surprising result: while the Lipschitz constant of $N$
	on $H^1(\R)$ clearly depends on the $H^1$ norm,  the difference
	$N(f_1)-N(f_2)$ is small whenever $f_1$ is close to $f_2$ in the
	weaker $L^2$ norm as soon as  $f_1 $ and $f_2$
	are bounded in $H^1(\R)$.
\end{remark}
\begin{proof}
 Recall the notation
  $a\vee b=\max(a,b)$, we also use $a\wedge b=\min(a,b)$ in the following.
 To  prove the first part recall
  $V(a)=\int_0^a P(s)\, \mathrm{d}s= \int_0^a h(s)s\, \mathrm{d}s$. Thus
  \begin{align*}
  	\big| V(|z|) - V(|w|) \big| &\le
  		\left| \int_{|w|}^{|z|} h(s)s\, \mathrm{d}s \right|
  		\lesssim  \int_{|z|\wedge|w|}^{|z|\vee|w|} (1+s^p)s\, \mathrm{d}s \\
  		&\lesssim \big( |z|+|w|+|z|^{p+1}+ |w|^{p+1} \big) |z-w|
  \end{align*}
  for all $z,w\in\C$. Thus,
  \begin{align} \label{eq:N Lipshitz-1}
  	\big| N(f_1) - N(f_2) \big|
  		&\le \iint_{\R^2} \big| V(|T_r f_1|)- V(|T_rf_2|)|dx\psi(r) dr \nonumber\\
		&\lesssim \iint_{\R^2} (|T_r f_1|+|T_rf_2|+|T_rf_1|^{p+1}+|T_rf_2|^{p+1})|T_r(f_1-f_2)|dx\psi(r)dr.
  \end{align}
  Using the Cauchy--Schwartz inequality in the $x$-integral and $T_r$ being unitary on $L^2$ for all $r\in \R$ one has
  \begin{align*}
  	\iint_{\R^2}  (|T_rf_1| & + |T_rf_2|)|T_r(f_1- f_2)|\, dx \psi(r)dr\\
	&\leq \int_\R(\|T_rf_1\|+\|T_rf_2\|)\|T_r(f_1-f_2)\|\psi(r)dr\\
  		&=(\|f_1\|+\|f_2\|) \|f_1- f_2\| \|\psi\|_{L^1}.
  \end{align*}
  Similarly, using the Cauchy--Schwarz inequality in the $x$-integral,
  then H\"older's inequality in the $r$-integral with exponents $\f{4}{p}$ and $\f{4}{4-p}$, and the Strichartz
  inequality with admissible pair $(2(p+1),\f{4(p+1)}{p})$, one has
  \begin{align*}
  	\iint_{\R^2} (|T_rf_1|^{p+1} & +  |T_rf_2|^{p+1})  |T_r(f_1-f_2)|\, dx \psi(r)dr \\
  		&\lesssim (\|f_1\|^{p+1} + \|f_2\|^{p+1}) \|f_1-f_2\|\|\psi\|_{L^{\frac{4}{4-p}}}.
  \end{align*}
  Substituting the last two bounds in \eqref{eq:N Lipshitz-1} proves \eqref{ineq:lipschitzL^2}.
 \smallskip

 To prove the second part note that now
 \begin{align*}
  	\big| V(|z|) - V(|w|) \big| &\le
  		\left| \int_{|w|}^{|z|} h(s)s\, \mathrm{d}s \right|
  		\le  \int_{|z|\wedge |w|}^{|z|\vee|w|} J_1(s)s\, \mathrm{d}s \\
  		&\le J_1(|z|\vee|w|) (|z|+|w|) |z-w|
 \end{align*}
  and using this in \eqref{eq:N Lipshitz-1} together with the Cauchy--Schwartz inequality  yields
  \begin{align*}
  	\big| N(f_1) &- N(f_2) \big| \\
  		&\lesssim \sup_{r\in\R} J_1(\|T_rf_1\|_{L^\infty}\vee \|T_rf_2\|_{L^\infty}) \iint_{\R^2}  (|T_rf_1|+|T_rf_2|)|T_rf_1- T_rf_2|\, dx \psi(r)dr \\
  		&\le  J_1(\|f_1\|_{H^1}\vee \|T_rf_2\|_{H^1})
  			(\|f_1\|+\|f_2\|)\|f_1-f_2\|\|\psi\|_{L^1},
  \end{align*}
which proves  \eqref{ineq:lipschitzH^1}.
\end{proof}

\bpf [Proof of Theorem \ref{thm:Stability result}]
We show the stability of the set of ground states adapting a proof
from \cite{HKS}, see also \cite{CL}.
We will first prove the positive average dispersion case.
%

Arguing by contradiction, assume that $S_\lambda^\dav$ is not stable. Then there exist $\veps_0>0$, a sequence $(\phi_n)_n$ in $H^1(\R)$ with
\beq  \label{not stable 0}
   d(\phi_n, S^\dav_\lambda)\coloneqq
 	\inf_{f\in S_\lambda^\dav}\|\phi_n -f\|_{H^1} < \frac{1}{n}
 	\quad\text{for } n\in\N\, ,
\eeq
and a sequence $(t_n)_n$ of times  such that
\beq \label{not stable}
  d(u_n(\cdot, t_n ), S^\dav_\lambda)\geq \veps_0
\eeq
for all $n$, where $u_n$ are solutions of \eqref{eq:main} with the initial data $\phi_n$.

We can then choose a sequence $(f_n)_n\subset S^\dav_\lambda$ such that
 $\|\phi_n -f_n\|_{H^1}< \frac{1}{n}$  for all $n\in\N$.
 Since $\|f_n\|^2=\lambda$ and $E(f_n)= E^\dav_\lambda$, the
 coercivity of the energy \eqref{eq:coercivity-2} shows that
 $\|f_n'\|$ is bounded, hence
 $(\phi_n)_n$ is a bounded sequence in $H^1(\R)$.
 In addition,
$$
\left|\|\phi_n\|-\lambda^{1/2}\right|=\bigl|\|\phi_n\|-\|f_n\|\bigr| \leq \|\phi_n-f_n\|_{H^1} \to 0 \quad \text{ as } n\to\infty\, ,
$$
so $\|\phi_n \|^2\to \lambda$ as $n\to \infty$. By mass conservation we also have $\|u_n(t)\|^2\to \lambda$ as $n\to\infty$ uniformly in $t\in\R$.

Moreover, $E(\phi_n)\to E_\lambda ^\dav$ as $n\to \infty$. Indeed, we have
\begin{align*}
  |E(\phi_n)-E_\lambda^\dav| =|E(\phi_n)-E(f_n)|
    	\leq \f{\dav}{2}\left|\|\phi_n'\|^2-\|f_n'\|	^2\right| +\left|N(\phi_n)-N(f_n)\right|.
\end{align*}
Using the reverse triangle inequality, we obtain
\begin{equation*}
	\left|\|\phi_n'\|^2-\|f_n'\|^2\right|
	= (\|\phi_n'\|+\|f_n'\|) \left|\|\phi_n'\|-\|f_n'\|\right|
	\le (\|\phi_n'\|_+\|f_n'\|) \|\phi_n'-f_n'\|
\end{equation*}
which together with
 $(\phi_n)_n$ and $(f_n)_n$ being bounded in $H^1(\R)$ and \eqref{ineq:lipschitzH^1} shows
 \begin{align*}
  &|E(\phi_n)-E_\lambda^\dav| \\
  &\lesssim (\|\phi_n'\|+\|f_n'\|) \|(\phi_n-f_n)'\|+J_1(\|\phi_n\|_{H^1}\vee\|f_n\|_{H^1})
				\left(\|\phi_n\|+ \|f_n\|\right)\|\phi_n-f_n\| \\
  &\lesssim  \|(\phi_n-f_n)'\|
			+	\|\phi_n-f_n\| \to 0
\end{align*}
as $n\to\infty$.

By energy conservation we also have $E\big(u_n(\cdot,t_n)\big)= E(\phi_n)\to E^\dav_\lambda$, i.e, it is an energy minimizing sequence, except that it might not have the correct $L^2$ norm. Since $(u_n(\cdot,t_n))_n$
is energy minimizing and its $L^2$ norm is bounded, the coercivity of the energy \eqref{eq:coercivity-2} implies that $(u_n(\cdot,t_n))_n$ is bounded in $H^1(\R)$.

To normalize the $L^2$ norm of $u_n(\cdot,t_n)$ let
$\alpha_n:=\lambda^{1/2}\|\phi_n\|^{-1}$ and set $g_n:=\alpha_n u_n(\cdot,t_n)$. From mass conservation it is clear that
$$
\|g_n\|^2=\alpha_n^2\| u_n(\cdot,t_n)\|^2=\alpha_n^2\| \phi_n\|^2=\lambda\, .
$$
Moreover, $\alpha_n \to 1$ as $n\to \infty$, since
$\|\phi_n\|^2\to \lambda$, so by  \eqref{ineq:lipschitzH^1} we also have, similarly as above,
\begin{align*}
	  |E(g_n)-E(u_n(\cdot,t_n))|
    &\leq \f{\dav}{2}\left|\|g_n'\|^2-\|u_n(\cdot,t_n)'\|	^2\right| +\left|N(g_n)-N(u_n(\cdot,t_n))\right|\\
    &\lesssim \|(g_n-u_n(\cdot,t_n))'\|+\|g_n-u_n(\cdot,t_n)\|\\
	 &\lesssim  |\alpha_n-1| \to 0 \text{ as } n\to \infty
\end{align*}
since $(g_n)_n$ and $(u_n(\cdot,t_n))_n$ are bounded in $H^1(\R)$.

Hence $(g_n)_{n\in\N}$ is a proper energy minimizing sequence.
The tightness result of \cite[Proposition 4.5]{ChoiHuLee2016},
more precisely, its extension to our more general setting,
see the second part of Remark \ref{rem:nonempty set of ground states},
shows us that there exists $K<\infty$ such that, for any $L>0$,
  \beq\label{eq:tightFourier}
    \sup_{n\in\N} \int_{|\eta|>L} |\widehat{g}_n(\eta)|^2\, d\eta \le \frac{K}{L^2} \notag
  \eeq
  where $\widehat{g}_n$ is the Fourier transform of $g_n$,
and that there exist shifts $y_n$ such that
$$
\lim_{R\to \infty} \sup _{n\in \N} \int _{|x|>R}|g_n(x-y_n)|^2 dx =0.
$$
Of course, the shifted sequence $\wti{g}_n= g_n(\cdot-y_n)$ is again
a minimizing sequence and thanks to the above bounds for $g_n$ it is tight
in the sense of measures. Since it is also bounded in $H^1(\R)$, there exist a subsequence, we still denote by $\wti{g}_n$ which converges weakly in
$H^1(\R)$ to some $\wti{g}\in H^1(\R)$, hence also weakly in $L^2(\R)$.
The tightness bounds above then imply that this subsequence also
converges strongly in $L^2(\R)$, see, for example, \cite[Lemma A.1]{HuLee2012} or \cite{Pego}.
Thus $\|\wti{g}\|^2=\lambda>0$ and since $\wti{g}_n$ is bounded in $H^1(\R)$
the inequality \eqref{ineq:lipschitzL^2} shows $N(\wti{g}_n)$ converges to
$N(\wti{g})$. Moreover, by weak convergence in $H^1(\R)$ we have
$\liminf_{n\to\infty} \|{\wti{g}_n}'\|^2 \ge  \|{\wti{g}}'\|^2 $, i.e,
the energy is lower semi--continuous under weak convergence. Since
$\wti{g}_n$ is an energy minimizing sequence, this yields
$E(\wti{g})= \lim_{n\to\infty} E(\wti{g}_n)= E^\dav_\lambda$.
Thus
$\lim_{n\to\infty} \|{\wti{g}_n}'\|^2 =  \|{\wti{g}}'\|^2$ and so
$\wti{g}_n$ \emph{converges strongly } in $H^1(\R)$ to $\wti{g}$.

Let $k_n= \wti{g}(\cdot+y_n)$. Then clearly
$k_n\in S_\lambda^\dav$ and
\begin{align*}
	\label{ineq:contradiction}
 \|u_n(\cdot,t_n) - k_n\|_{H^1}
 	&\leq \|u_n(\cdot,t_n) -g_n\|_{H^1} +\|g_n -k_n\|_{H^1}\\
 	&=|1-\alpha_n|\|u_n\|_{H^1}+\|\widetilde{g}_n-\widetilde{g}\|_{H^1}
 		\to 0
\end{align*}
as $n\to \infty$, which contradicts \eqref{not stable}.
Thus $S^\dav_\lambda$ is orbitally stable if $\dav>0$.

\smallskip

Now we come to the proof of orbital stability of $S^0_\lambda$
for saturated nonlinearities, i.e., under assumption \ref{ass:A5},
when the average dispersion is zero.
We again argue by contradiction and assume that there exist
$\veps_0>0$ and a sequence $(\phi_n)_{n\in\N}$ such that \eqref{not stable 0}
and \eqref{not stable} hold. We then choose a sequence
$f_n\in S^0_\lambda$ with $\lim_{n\to\infty}\|\phi_n-f_n\|=0$.
Arguing as in the case of $\dav>0$ we also have
$\lim_{n\to\infty}\|\phi_n\|^2=\lambda$ and, by mass conservation,
$\|u_n(t)\|^2\to \lambda$ as $n\to\infty$ uniformly
in $t\in\R$.
Moreover, since $|h(a)|\lesssim a^{p_1} + a^{p_2}$ with
$1\le p_1\le p_2<3 $ the bound \eqref{ineq:lipschitzL^2}
still applies. Hence
\begin{equation*}
  |E(\phi_n)-E_\lambda^0| =|E(\phi_n)-E(f_n)|
    	 =  |N(\phi_n) -N(f_n)|
    \lesssim 
   	\|\phi_n-f_n\| \to 0
\end{equation*}
as $n\to\infty$, since $(\phi_n)_n$ and $(f_n)_n$ are bounded in
$L^2(\R)$. Thus $E(\phi_n)\to E_\lambda ^0$ as $n\to \infty$ and
$E\big(u_n(\cdot,t_n)\big)= E(\phi_n)\to E^0_\lambda$, again  by
energy conservation. So as before $(u_n(t_n))$ is an energy
minimizing sequence, except that it might have the proper normalization
only in the limit  of large $n$.  We again normalize
this sequence, setting $g_n:=\alpha_n u_n(\cdot,t_n)$, with
$\alpha_n:=\lambda^{1/2}\|\phi_n\|^{-1}$, which converges to $1$
in the limit $n\to\infty$.
We follow the previous arguments, in the case $\dav>0$, to see that
$(g_n)_{n\in\N}$ is again an energy minimizing sequence which
is properly normalized in $L^2$.

However, at this stage we have to deviate from the arguments for
$\dav>0$, because we only know that the sequence $(g_n)_n$ is
normalized in $L^2$, and, unlike the case $\dav>0$, we do not have
any additional information about $(g_n)_n$ at this stage. In
particular, we do not know whether $\|g_n\|_\infty\le C$ uniformly in
$n$ for some constant $C<\infty$, which allowed us to use a
modification of the tightness results from \cite{ChoiHuLee2016} when
$\dav>0$, see Remark \ref{rem:existence-ground-states}.

To get around this dilemma, we note that \cite[Lemma 3.10]{HLRZ}
shows that given the minimizing sequence $(g_n)$ there exists
another minimizing sequence $(h_n)_{n\in \N}\subset L^2(\R)\cap L^\infty(\R)$ with
\begin{align*}
	\sup_{r\in \supp(\psi)} \|T_r h_n\|_{L^\infty}\leq C_\lambda \, .
\end{align*}
From the construction of this modified minimizing sequence, see the
proof of Lemma 3.10 in \cite{HLRZ}, we also know that
\beq \label{conv:gh}
\|g_n- h_n\|\to 0
\eeq
as $n\to \infty$. This allows us to apply a modification of the
tightness result \cite[Proposition 4.6]{ChoiHuLee2016} for
the sequence $(h_n)_n$, see
Remark \ref{rem:existence-ground-states}.
This yields shifts $y_n$ and boosts $\xi_n$ such that
\begin{align}
	\lim_{R\to \infty} \sup _{n\in \N} \int _{|x|>R}|h_n(x-y_n)|^2 dx &=0
	\label{eq:tightness real space}
\intertext{and}
    \lim _{L\to \infty}\sup_{n\in\N} \int_{|\eta|>L} |\widehat{h}_n(\eta-\xi_n)|^2\, d\eta &=0
    \label{eq:tightness Fourier space}
\end{align}
where $\widehat{h}_n$ is the Fourier transform of $h_n$.

Let $\wti{h}_n= e^{i\xi_n \cdot}h_n(\cdot-y_n), n \in \N$ be the
shifted and boosted sequence. Then $(\wti{h}_n)$ is again a
minimizing sequence with $\|\wti{h}_n\|_{L^2}^2=\lambda$, which is
bounded in $L^\infty$  and it is tight, i.e., the
bounds \eqref{eq:tightness real space} and
\eqref{eq:tightness Fourier space} hold with  $y_n=\xi_n=0$ and
$h$ replaced by $\wti{h}$. Using a weakly convergent
subsequence, also denoted by $\wti{h}_n$, the tightness of
$\wti{h}_n$ yields strong convergence of this subsequence, see, e.g.,
\cite[Lemma A.1]{HuLee2012} or \cite{Pego}, to some $\wti{h}$.
Using $L^2$ continuity of the energy when $\dav=0$, this function
$\wti{h}$ is a minimizer of the energy, i.e., $\wti{h}\in S^0_\lambda$,
and so are the shifted and boosted functions
$k_n= e^{-i\xi_n(\cdot+y_n)}\wti{h}(\cdot+y_n)$, where we unravel
the boosts and shifts which lead from $h_n$ to $\wti{h}_n$.
By construction, $\|h_n-k_n\|\to 0$ for $n\to\infty$.
Using
that $0<\veps_0\le \|u_n(\cdot,t_n)- f\|$ for any $f\in S^0_\lambda$
we get the contradiction
\begin{align*}
	0<\veps_0& \le \|u_n(\cdot,t_n)- k_n\|
			\le  \|u_n(\cdot,t_n)- g_n\| + \|g_n-h_n\| + \|h_n-k_n\| \\
		&= |1-\alpha_n| \|u_n\|  + \|g_n-h_n\| + \|h_n-k_n\|
			\to 0 \text{ for } n\to\infty.
\end{align*}
Hence $S^0_\lambda$ is orbitally stable, even for
saturating nonlinearities.

\smallskip
For non--saturated nonlinearities, the proof for the zero
average dispersion case is analogous to that for saturating nonlinearities, except that one can directly use
the tightness result from Proposition 4.6 in \cite{ChoiHuLee2016}
and does not have to modify the minimizing sequence to make it
bounded in $L^\infty$.
\end{proof}

\bigskip

\noindent
\textbf{Acknowledgements: }
Young--Ran Lee and Mi--Ran Choi are supported by the National Research Foundation of Korea (NRF) grants funded by the Korean government (MSIT) NRF-2020R1A2C1A01010735 and (MOE) NRF-2021R1I1A1A01045900.
Dirk Hundertmark is funded by the Deutsche Forschungsgemeinschaft (DFG, German Research Foundation) -- Project-ID 258734477 -- SFB 1173.
\renewcommand{\thesection}{\arabic{chapter}.\arabic{section}}
\renewcommand{\theequation}{\arabic{chapter}.\arabic{section}.\arabic{equation}}
\renewcommand{\thetheorem}{\arabic{chapter}.\arabic{section}.\arabic{theorem}}

\def\cprime{$'$}


\begin{thebibliography}{100}
\small{

\bibitem{AB98}
	M.~J.~Ablowitz and G.~Biondini,
	\textit{Multiscale pulse dynamics in communication systems with strong dispersion management.}
	Optics Letters \textbf{23} (1998), 1668--1670.
	\hfill

\bibitem{AK} J.~Albert and E.~Kahlil, \textit{On the well--posedness of the Cauchy problem for some nonlocal nonlinear Schr\"{o}dinger equations}.
	Nonlinearity\ \textbf{30} (2017), 2308--2333.
    \hfill
	
\bibitem {AR}
	A.~Ambrosetti and P.~H.~Rabinowitz, \textit{Dual variational methods in critical point theory and applications.}
	J. Funct. Anal. \textbf{14} (1973), 349--381.
	\hfill

\bibitem{AHH}
	I.~Anapolitanos, M.~Hott, and D.~Hundertmark,
	\textit{Derivation of the Hartree equation for compound Bose gases
	in the mean field limit.}
	Rev.\ Math.\ Phys.\ \textbf{29} (2017), no.~7, 1750022, 28 pp.
    \hfill

\bibitem{ASS}
 P.~Antonelli, J.~C.~Saut, and C.~Sparber, \textit{Well-posedness and averaging of NLS with time-
periodic dispersion management.} Adv. Differential Equations \textbf{18} (2013), 49--68.
\hfill


\bibitem{Ball} J.~M.~Ball,
	\textit{Strongly Continuous Semigroups, Weak Solutions, and the Variation of Constants Formula}.
		Proc.\ Amer.\ Math.\ Soc.\ \textbf{63} (1977), no.~2, 370--373.
	\hfill

\bibitem{Cazenave} T.~Cazenave, \textit{Semilinear Schr\"odinger Equations},
	Courant Lecture Notes in Mathematics, \textbf{10}. New York University, Courant Institute of Mathematical Sciences, New York; American Mathematical Society, Providence, RI, 2003. xiv+323 pp. \hfill

\bibitem{CL} T.~Cazenave and P.~L.~Lions, \textit{Orbital stability if standing waves for some nonlinear Schrödinger Equations}.  Comm.\ Math.\ Phys.\ \textbf{85} (1982), 549--561.
	\hfill
	
\bibitem{ChoiHuLee2016}
M.--R.\ Choi, D.~Hundertmark, and Y.--R.~ Lee,
 \textit{ Thresholds for existence of dispersion management solitons for general nonlinearities}. 
	SIAM  J.\  Math.\ Anal. \textbf{49} (2017), no.~2, 1519--1569.
	\hfill

\bibitem{CKL}
M.--R.~Choi, Y.~Kang, and Y.--R.~Lee,
 \textit{On dispersion managed nonlinear {S}chr\"{o}dinger equations with
  lumped amplification}.
 J. Math. Phys. \textbf{62} (2021), no.~7, 071506, 1--16.
\hfill

\bibitem{CLA}
M.--R.~Choi and Y.--R.~Lee.
 \textit{Averaging of dispersion managed nonlinear schr\"odinger equations}.
 Nonlinearity \textbf{35} (2022), no.~4, 2121--2133.
\hfill

\bibitem{EHL2009} M.~B.~Erdo\smash{\u g}an, D.~Hundertmark, and Y.--R.~Lee,
    \textit{Exponential decay of dispersion managed solitons for vanishing average dispersion.}
    Math.\ Res.\ Lett.\ \textbf{18} (2011), no.~1, 13--26.
    \hfill

\bibitem{GT96a} I.~Gabitov and S.~K.~Turitsyn,
    \textit{Averaged pulse dynamics in a cascaded transmission system
    with passive dispersion compensation.}
    Opt.\ Lett.\  \textbf{21} (1996), 327--329.
    \hfill

\bibitem{GT96b} I.~Gabitov and S.~K.~Turitsyn,
    \textit{Breathing solitons in optical fiber links.}
    JETP Lett.\ \textbf{63} (1996), 861.
    \hfill

\bibitem{GV} J.~Ginibre and G.~Velo,
	\textit{The global Cauchy problem for the nonlinear Schr\"odinger equation revisited}.
	Ann. Inst. H. Poincar\'e Anal. Non Lin\'eaire \textbf{2} (1985), 309--327.
	\hfill


\bibitem{GH} W.~Green and D.~Hundertmark, \textit{Exponential Decay for dispersion managed solitons for general dispersion profiles}.
	Lett.\ Math.\ Phys.\ \textbf{106} (2016), no.\ 2, 221--249.
	\hfill

\bibitem{HKS} D.~Hundertmark, P.~Kunstmann, and R.~Schnaubelt, \textit{Stability of dispersion managed solitons for vanishing average dispersion}.
	Arch.\ Math.\ \textbf{104} (2015), no.\ 3, 283--288.
	\hfill

\bibitem{HuLee2009} D.~Hundertmark and Y.--R.~Lee,
     \textit{Decay estimates and smoothness for solutions of the dispersion managed   non--linear Schr\"odinger equation.}
     Comm.\ Math.\ Phys.\ \textbf{286} (2009), no.~3, 851--873.
     \hfill

\bibitem{HuLee2012}
D.~Hundertmark and Y.--R.~Lee, \textit{ On non--local variational problems with lack of compactness related to non--linear optics.} J. Nonlinear Sci.\ \textbf{22} (2012),  1--38.
\hfill

\bibitem{HLRZ}
	D.~Hundertmark, Y.--R.~Lee, T.~Ried, and V.~Zharnitsky,
	\textit{Solitary waves in nonlocal NLS with dispersion averaged saturated nonlinearities.}
	 J. Differential Equations \textbf{265} (2018), no.~8,  3311--3338.
	\hfill

\bibitem{Kato1987}
T.~Kato, \textit{On nonlinear Schr\"odinger  equations.} Ann. Inst. H. Poincare Phys. Theor.\ \textbf{46}  (1987), 113--129.
	\hfill

\bibitem{KT} M.~Keel and T.~Tao,
	\textit{Endpoint Strichartz estimates.}
	Amer. J. Math. \textbf{120} (1998), no~ 5, 955--980.
	\hfill

\bibitem{LandauLifshitz}
    L.~D.~Landau and E.~M.~Lifshitz,
    \textit{Course of theoretical physics. Vol. 1. Mechanics.}
    Third edition. Butterworth-Heinemann, Oxford-New York-Toronto, Ont., 1976.
	\hfill
	
\bibitem{Murphy-review}
	 J.~Murphy, \textit{A review of modified scattering for the 1d
	cubic NLS.} Harmonic analysis and nonlinear partial differential equations,
	RIMS K\^{o}ky\^{u}roku Bessatsu \textbf{B88} (2021), 119–146.
\hfill

\bibitem{MH1}
	J.~Murphy and T.~V.~Hoose, \textit{Modified scattering for a dispersion--managed nonlinear Schr\"odinger equation.} NoDEA Nonlinear Differential Equations Appl. \textbf{29} (2022), no. 1, 1--11
\hfill

\bibitem{MH2}
	J.~Murphy and T.~V.~Hoose, \textit{Well--posedness and blowup for the dispersion-managed nonlinear Schr\"odinger equation.}
arXiv:2110.08372 \hfill



 \bibitem{Ozawa}
 	 T.~Ozawa, \textit{Remarks on proofs of conservation laws for nonlinear Schr\"odinger equations.} Calc. Var. Partial Differential Equations \textbf{25} (2006), 403--408.	
 \hfill


\bibitem{Pego}
	R.~L.~Pego, \textit{Compactness in $L^2$ and the Fourier transform.}
	Proc.\ Amer.\ Math.\ Soc.\ \textbf{95} (1985), no.\ 2, 252--254.
	\hfill

\bibitem{Stanislavova05}
	M.~Stanislavova,
    \textit{Regularity of groundstate solutions of dispersion managed nonlinear schr\"odinger equations.}
   J. Differential Equations\ \textbf{210} (2005), no.\ 1, 87--105.
	\hfill

\bibitem{Strichartz}
    R.~S.~Strichartz,
    \textit{Restrictions of Fourier transforms to quadratic surfaces and
        decay of solutions of wave equations.}
    Duke Math.\ J.\ \textbf{44} (1977), 705--714.
    \hfill

\bibitem{Tao}
	T.~Tao, \textit{Nonlinear Dispersive Equations. Local and Global Analysis.}
	CBMS Regional Conference Series in Mathematics, 106. Published for the Conference Board of the Mathematical Sciences, Washington, DC; by the American Mathematical Society, Providence, RI, 2006. xvi+373 pp.
	\hfill

\bibitem{TBF} S.~K.~Turitsyn, B.~Bale, and M.~P.~Fedoruk,
	\textit{Dispersion-managed solitons in fibre systems and lasers.}
	Physics Reports, \textbf{521} (2012), no.~4, 135--203.
	\hfill

\bibitem{ZGJT01} V.~Zharnitsky, E.~Grenier, C.~K.~R.~T.~Jones, and S.~K.~Turitsyn,
\textit{Stabilizing effects of dispersion management.}
    Phys. D.\ \textbf{152-153} (2001), 794--817.
    \hfill

}
\end{thebibliography}
\end{document}